\documentclass[12pt]{amsart}
\title{\'Ez fields}
\author{Erik Walsberg}
\address{Department of Mathematics\\ University of California, Irvine}
\email{ewalsber@uci.edu}
\urladdr{https://www.math.uci.edu/\textasciitilde ewalsber}

\author{Jinhe Ye}
\address{Institut de math\'ematiques de Jussieu – Paris Rive Gauche}
\email{jinhe.ye@imj-prg.fr}
\urladdr{https://sites.google.com/view/vincentye}
\usepackage{amsmath, amssymb, amsthm, enumitem, comment}
\usepackage{tikz-cd}
\usepackage[mathscr]{euscript}
\usepackage{fullpage} 	
\usepackage{hyperref}
\usepackage{lineno}
\usepackage[all]{xy}
\usepackage{centernot}
\usepackage[Symbol]{upgreek}
\usepackage{xcolor}
\usepackage[normalem]{ulem}

\DeclareFontFamily{U}{BOONDOX-calo}{\skewchar\font=45 }
\DeclareFontShape{U}{BOONDOX-calo}{m}{n}{
  <-> s*[1.05] BOONDOX-r-calo}{}
\DeclareFontShape{U}{BOONDOX-calo}{b}{n}{
  <-> s*[1.05] BOONDOX-b-calo}{}
\DeclareMathAlphabet{\mathcalboondox}{U}{BOONDOX-calo}{m}{n}
\SetMathAlphabet{\mathcalboondox}{bold}{U}{BOONDOX-calo}{b}{n}
\DeclareMathAlphabet{\mathbcalboondox}{U}{BOONDOX-calo}{b}{n}

\DeclareMathOperator*{\forkindep}{\raise0.2ex\hbox{\ooalign{\hidewidth$\vert$\hidewidth\cr\raise-0.9ex\hbox{$\smile$}}}}

\newcommand{\qsolve}{\Qq_{\mathrm{solv}}}
\newcommand{\fr}{\operatorname{Fr}}
\newcommand{\frrel}{\operatorname{Fr}_{V/W}}
\newcommand{\Sa}[1]{\ensuremath{\mathscr{#1}}}

\newcommand{\pac}{\mathrm{PAC}}

\newcommand{\kalg}{K^{\mathrm{alg}}}

\newcommand{\Gal}{\operatorname{Gal}}
\newcommand{\Spec}{\operatorname{Spec}}

\newcommand{\Chara}{\operatorname{Char}}

\newcommand{\red}{\mathrm{red}}

\newtheorem*{claim-star}{Claim}
\newtheorem{theorem}{Theorem}[section] 
\newtheorem{lemma}[theorem]{Lemma}

\newtheorem{prop-def}[theorem]{Proposition-Definition}
\newtheorem{corollary}[theorem]{Corollary}
\newtheorem{fact}[theorem]{Fact}
\newtheorem{fact-eh}[theorem]{Fact(?)}

\newtheorem{proposition}[theorem]{Proposition}
\newtheorem{proposition-eh}[theorem]{Proposition(?)}
\newtheorem*{theorem-star}{Theorem}
\newtheorem*{conjecture-star}{Conjecture}
\newtheorem*{question-star}{Question}
\newtheorem*{lemma-star}{Lemma}

\newtheorem*{thmA}{Theorem A}
\newtheorem*{factA}{Fact A}
\newtheorem*{corA}{Corollary A}
\newtheorem*{thmB}{Theorem B}
\newtheorem*{thmC}{Theorem C}
\newtheorem*{thmD}{Theorem D}
\newtheorem*{thmE}{Theorem E}

\newtheorem*{corD}{Corollary D}
\newtheorem*{thmF}{Theorem F}

\theoremstyle{definition}

\newtheorem{example}[theorem]{Example}

\theoremstyle{remark}
\newtheorem{claim}[theorem]{Claim}

\newcommand{\Aa}{\mathbb{A}}

\newcommand{\Qq}{\mathbb{Q}}

\newcommand{\Zz}{\mathbb{Z}}
\newcommand{\Nn}{\mathbb{N}}
\newcommand{\Cc}{\mathbb{C}}

\newcommand{\cE}{\mathscr{E}}

\newcommand{\meno}{\medskip \noindent}

\newenvironment{claimproof}[1][\proofname]
               {
                 \proof[#1]
                 
               }
               {
                 \endproof
               }

\begin{document}
\maketitle

\begin{abstract}
Let $K$ be a field.
The \'etale open topology on the $K$-points $V(K)$ of a $K$-variety $V$ was introduced in \cite{firstpaper}.
The \'etale open topology is non-discrete if and only if $K$ is large.
If $K$ is separably, real, $p$-adically closed then the \'etale open topology agrees with the Zariski, order, valuation topology, respectively.
We show that existentially definable sets in perfect large fields behave well with respect to this topology: such sets are finite unions of \'etale open subsets of Zariski closed sets.
This implies that existentially definable sets in arbitrary perfect large fields enjoy some of the well-known topological properties of definable sets in algebraically, real, and $p$-adically closed fields.
We introduce and study the class of \'ez fields: $K$ is \'ez if $K$ is large and every definable set is a finite union of \'etale open subsets of Zariski closed sets.
This should be seen as a generalized notion of model completeness for large fields.
Algebraically closed, real closed, $p$-adically closed, and bounded $\mathrm{PAC}$ fields are \'ez.
(In particular pseudofinite fields and infinite algebraic extensions of finite fields are \'ez.)
We develop the basics of a theory of definable sets in \'ez fields.
This gives a uniform approach to the theory of definable sets across all characteristic zero local fields and a new topological theory of definable sets in bounded $\mathrm{PAC}$ fields.
We also show that some prominent examples of possibly non-model complete model-theoretically tame fields (characteristic zero $t$-Henselian fields and Frobenius fields) are \'ez.
\end{abstract}

\meno
\textbf{Throughout $K$ is a field.}
We are concerned with two properties of $K$: largeness and logical tameness.
We first recall largeness, which we view as a field-arithmetical tameness notion.
Recall that $K$ is \textbf{large} if every $K$-curve with a smooth $K$-point has infinitely many $K$-points.
Largeness was introduced by Florian Pop~\cite{pop-embedding} for Galois-theoretic purposes and has been studied under multiple names.
Separably closed fields, real closed fields, Henselian fields (i.e. fields which admit non-trivial Henselian valuations), quotient fields of Henselian domains\footnote{Such fields may not be Henselian, e.g. $\Cc[[x,y]]$ is a Henselian domain whose fraction field is not a Henselian field.}, pseudofinite fields, infinite algebraic extensions of finite fields, $\pac$ fields, $p$-closed fields, and fields which satisfy a local-global principle are all large.
Finite fields, number fields, and function fields are not large, hence fields that are finitely generated over their prime subfields are not large.
In particular local fields are large and global fields are not.

\meno
``Logical tameness" does not admit a precise definition.
It is a remarkable empirical fact that exactly one of following holds in all fields $K$ whose theories are understood:
\begin{enumerate}
\item $K$ interprets the theory of the ring $\Zz$.
\item Every formula in the language of rings is equivalent to a ``simple" formula over $K$.
\end{enumerate}
We emphasize that we do not expect this dichotomy to hold for arbitrary fields, there should be all kinds of unnatural fields in between. In this paper, we generally consider fields satisfying (2) as ``logically tame". In practice one establishes (2) by showing that $K$ is model complete in some ``reasonable" expansion of the language of rings.
We do not have a precise definition of ``logical tameness" as there does not seem to be a definition that captures the notion of a ``reasonable expansion of the language of rings".
A second remarkable empirical fact is that all of the logically tame fields we know are large.
Again, we do not expect this to hold hold for arbitrary fields.
There may be strange non-large logically tame examples beyond the fields we know.

\meno
In \cite{firstpaper} we introduced a topology over an arbitrary field $K$ which is non-discrete if and only if $K$ is large.
We show that if $K$ is perfect and large then existentially definable sets behave well with respect to our topology, hence if $K$ is model complete in the language of rings then all definable sets are well-behaved with respect to the topology.
Hence the assumption (made precise below) that $K$-definable sets are well-behaved with respect to this topology is a topological generalization of model completeness.
We make the necessarily vague conjecture that definable sets in all known logically tame perfect fields are well-behaved with respect to our topology and go some distance towards proving this conjecture.

\meno
Let $V$ be a $K$-variety and $V(K)$ be the set of $K$-points of $V$.
The \textbf{\'etale open} ($\Sa E_K$-) topology on $V(K)$ of $V$ is the topology with basis given by sets of the form $f(W(K))$ for \'etale morphisms $f\colon W\to V$. More details can be found in \cite{firstpaper}.
The field $K$ is large if and only if the $\Sa E_K$-topology on $K = \Aa^1(K)$ is not discrete if and only if the $\Sa E_K$-topology on $V(K)$ is non-discrete whenever $V(K)$ is infinite.
The \'etale open topology over a separably closed, real closed, and non-separably closed Henselian field agrees with the Zariski, order, and valuation topology, respectively.
In particular the \'etale open topology over a local field other than $\Cc$ agrees with the usual locally compact topology.
The \'etale open topology agrees with the Zariski topology if and only if $K$ is finite or separably closed.

\medskip
We define an
\textbf{\'ez subset} of $V(K)$ to be a finite union of definable \'etale open subsets of Zariski closed subsets of $V(K)$.
By Lemma~\ref{lem:ez-equiv} below a definable subset of $V(K)$ which is a finite union of \'etale open subsets of Zariski closed sets is \'ez.
Equivalently: an \'ez set is a definable set which is a finite union of sets which are locally Zariski closed in the $\Sa E_K$-topology.
Note that an \'ez subset of $K$ is a union of a definable \'etale open set and a finite set.

\meno
We let $\Sa Z$ be the collection of finite unions of Zariski open subsets of Zariski closed sets.
A subset of $K^m$ is quantifier free definable if and only if it is in $\Sa Z$.
Thus quantifier elimination for algebraically closed fields is equivalent to the following geometric statement:

\begin{factA}
Suppose $K$ is algebraically closed, $f \colon V \to W$ is a morphism of $K$-varieties, and $X \subseteq V(K)$ is in $\Sa Z$.
Then $f(X)$ is also in $\Sa Z$.
\end{factA}

Macintyre~\cite{Macintyre-omegastable} showed that an infinite field with quantifier elimination is algebraically closed, so Fact A fails when $K$ is not algebraically closed.
Theorem A generalizes Fact A as the \'etale open topology over an algebraically closed field agrees with the Zariski topology.

\begin{thmA}
Suppose that $K$ is large and perfect and $f \colon V \to W$ is a morphism of $K$-varieties.
If $X$ is an \'ez subset of $V(K)$ then $f(X)$ is an \'ez-subset of $W(K)$.
\end{thmA}

If $K$ is not large then the conclusion of Theorem A trivially holds.
If $K$ is large, imperfect, and of characteristic $p$, then the conclusion of Theorem A fails as the set of $p$th powers is not an \'ez set, see Section~\ref{section:sharp}.
Theorem A immediately implies Corollary A.

\begin{corA}
Suppose $K$ is large and perfect.
Then any existentially definable subset of any $K^m$ is an \'ez set.
In particular any existentially definable subset of $K$ is a union of a definable \'etale open subset of $K$ and a finite set.
\end{corA}

This prompts us to prove some general facts on \'ez sets.
We show that certain properties of definable sets in algebraically closed fields generalize to \'ez sets in large perfect fields.
If $K$ is not large then any subset of $V(K)$ is trivially \'etale open, so largeness is the minimal requirement necessary for a theory of \'ez sets.
We let $\dim X$ be the dimension of the Zariski closure of a subset $X$ of $V$.
If $X \subseteq K^m$ then $\dim X$ is the maximal number of polynomial functions on $X$ that can be algebraically independent over $K$~\cite[Example 2.12.3]{lou-dimension}.

\begin{thmB}
Suppose that $K$ is large and perfect, $V$ is a smooth irreducible $K$-variety, and $X,Y$ are nonempty \'ez subsets of $V(K)$.
Then
\begin{enumerate}
\item There are pairwise disjoint smooth irreducible subvarieties $V_1,\ldots,V_k$ of $V$ and $X_1,\ldots,X_k$ such that each $X_i$ is a definable \'etale open subset of $V_i(K)$ and $X = X_1\cup\cdots\cup X_k$.
\item $\dim X = \dim V$ if and only if $X$ has nonempty $\Sa E_K$-interior in $V(K)$.
\item if $X \subseteq Y$ and $\dim X = \dim Y$ then $X$ has nonempty $\Sa E_K$-interior in $Y$.
\item There is a smooth subvariety $W$ of $V$, a nonempty \'etale open subset $O$ of $W(K)$, and a dense open subvariety $U$ of $V$ such that $O = X \cap U$ and $\dim X \setminus O < \dim X$.
\end{enumerate}
\end{thmB}

We say that $K$ is an \textbf{\'ez field} if $K$ is large and every definable set is an \'ez set.
We view this as a topological generalization of model completeness in the class of perfect large fields.
We will see that \'ez fields are perfect and that many of the known model-theoretically tame fields are \'ez.
We say that $K$ is \textbf{model complete} if $K$ is model complete in the language of rings and is \textbf{model complete by constants} if $K$ is model complete after some collection of constants is added to the language of rings.

\begin{thmC}
Suppose that one of the following holds:
\begin{enumerate}
\item $K$ is large and model complete,
\item $K$ is large, perfect, and model complete by constants,
\item $K$ is $t$-Henselian of characteristic zero, or
\item $K$ is a perfect Frobenius field.
\end{enumerate}
Then $K$ is \'ez.
\end{thmC}

Model complete fields are perfect\footnote{If $K$ is imperfect then the Frobenius $K \to K$ is not an elementary embedding.}, so $(1)$ and $(2)$ are immediate from Corollary A.
$(4)$ is proven in Section~\ref{section:frobenius}.
We recall t-Henselianity, a topological generalization of Henselianity introduced by Prestel-Zieger~\cite{Prestel1978}.
Suppose that $\uptau$ is a non-discrete field topology on $K$.
Then $X\subseteq K$ is \textbf{bounded} if for every neighbourhood $U$ of zero there is $\alpha\in K^\times$ such that $\alpha X \subseteq U$.
A field topology $\uptau$ on $K$ is \textbf{t-Henselian} if:
\begin{enumerate}
\item $\uptau$ is not discrete,
\item $(K\setminus U)^{-1}$ is bounded for any neighbourhood $U$ of zero, and
\item  for any $n$ there is an open neighborhood $U$ of zero such that if $\alpha_0,\ldots,\alpha_{n} \in U$ then $t^{n + 2} + t^{n + 1} + \alpha_{n} t^{n} + \cdots + \alpha_1 t + \alpha_0$ has a root in $K$.
\end{enumerate}
The field $K$ is t-Henselian if and only if $K$ admits a t-Henselian field topology, which must be unique if $K$ is not separably closed.
If $\uptau$ is induced by a non-trivial Henselian valuation on $K$ then $\uptau$ is t-Henselian, so a Henselian field is t-Henselian.
The order topology on a real closed field is t-Henselian.
The Henselian case of $(3)$ follows from known results on Henselian fields and the general case follows from the Henselian case by elementary transfer, see Section~\ref{section:Henselian}.

\meno
We discuss sharpness of Theorem E.
There are large perfect fields which are not \'ez, see Section~\ref{section:sharp}.
As mentioned above, if $K
$ is large and imperfect then set of $p$th powers is not an \'ez-set, hence $K$ is not \'ez.
There are large fields which are model complete by constants and imperfect, e.g. any imperfect separably closed field is model complete after constants naming a $p$-basis are added~\cite{delon-scf}.
Hence the assumption of perfection in (2) is necessary.

\meno
We first describe our other results on \'ez fields, then we discuss specific examples of \'ez fields below.
Following van den Dries~\cite{lou-dimension} we say that $K$ is \textbf{algebraically bounded} if for every definable $X \subseteq K^m \times K$ there are polynomials $f_1,\ldots,f_k \in K[x_1,\ldots,x_m,t]$ such that if $\alpha\in K^m$ and $X_\alpha = \{\beta \in K : (\alpha,\beta) \in X \}$ is finite then $X_\alpha\subseteq\{ \beta \in K : f_i(\alpha,\beta) = 0 \}$ for some $i \in \{1,\ldots,k\}$ such that $f_i(\alpha,t)$ is not constant zero.
Van den Dries showed that characteristic zero Henselian fields are algebraically bounded~\cite{lou-dimension}.
Jarden showed that perfect Frobenius fields are algebraically bounded~\cite{Jarden-bounded}, which is later generalized to perfect PAC fields by Chatzidakis and Hrushovski~\cite{CH-bounded}.
Junker and Koenigsmann showed that if $K$ is large and model complete then model-theoretic algebraic closure in $K$ agrees with field-theoretic algebraic closure~\cite{JK-slim}.
This property implies elimination of $\exists^\infty$ by~\cite[Theorem 2.5]{dim-fields}. Hence it implies algebraic boundedness.
We prove Theorem D.

\begin{thmD}
\'Ez fields are algebraically bounded.
\end{thmD}

Algebraically bounded fields are geometric (i.e. they eliminate $\exists^\infty$ and model-theoretic algebraic closure satisfies the exchange property) and the resulting notion of dimension agrees with algebraic dimension.
Corollary D follows, see \cite{lou-dimension} for details.

\begin{corD}
Suppose that $K$ is \'ez, $X$ is a definable subset of $K^m$, and $f$ is a definable function $X \to K^n$.
Then
\begin{enumerate}
\item  $Y_d := \{ \alpha \in K^n : \dim f^{-1}(\alpha) = d \}$ is definable for all $0 \le d \le n$, and
\item $\dim X = \max\{ d + \dim Y_d : 0 \le d \le n \}$.
\end{enumerate}
In particular $\dim f(X) \le \dim X$.
\end{corD}

If $\Chara(K) = p$ and $c \in K$ is not a $p$th power, then the map $K^2 \to K$, $(\alpha,\beta) \mapsto \alpha^p + c\beta^p$ is injective.
Hence algebraically bounded fields are perfect.

\meno
In Section~\ref{section:functions} we apply Theorems C and E to show that definable functions are generically continuous in \'ez fields.

\begin{thmE}
Suppose that $K$ is \'ez and $f \colon K^m \to K^n$ is definable.
Then $f$ is $\Sa E_K$-continuous on a dense Zariski open subset of $K^m$.
\end{thmE}

This gives a uniform proof that definable functions in characteristic zero local fields are generically continuous.
Theorem E follows from Proposition~\ref{prop:one var function decomp}, a more precise result on definable $K$-valued functions.

\subsection*{Examples of \'ez fields}
See \cite{EP-value} for an account of Henselianity.
Examples of characteristic zero Henselian fields are $\Qq_p$, algebraic extensions of $\Qq_p$, and the fields of Laurent series $L((t))$ and Puiseux series $L\langle\langle t \rangle\rangle$ over an arbitrary characteristic zero field $L$.

\meno
Algebraically and real closed fields are model complete by classical work of Tarski~\cite[Theorem 2.7.2, 2.7.3]{Hodges}.
Macintyre showed that $\Qq_p$ is model complete~\cite{Mac76}.
Model completeness of finite extensions of $\Qq_p$ follows from work of Prestel and Roquette~\cite[Theorem 5.1]{Prestel-roquette}.
Hence every characteristic zero local field is model complete. If $L$ is a model complete field of characteristic zero, by induction and Ax-Kochen-Ershov, $L((t_1))((t_2))\ldots((t_n))$ is model complete by constants. Thus for $L$ algebraically closed of characteristic zero, real closed, or $p$-adically closed, $L((t_1))((t_2))\ldots((t_n))$ is \'ez. In the mixed characteristic case, Derakhshan and Macintyre~\cite{Derakhshan2016ModelCF} showed that if $(K,v)$ is a finitely ramified Henselian valued field with value group a $\Zz$-group and model complete residue field, then $K$ is model complete. In particular, this shows any infinite algebraic extension of $\Qq_p$ with finite ramification is model complete.

\meno
We now discuss perfect $\pac$ fields which are model complete by constants.
See \cite[Chapter 11]{field-arithmetic} for an overview of $\pac$ fields.
Let $\Gal_K$ be the absolute Galois group of $K$.
Recall that $K$ is \textbf{bounded} if $K$ has only finitely many separable extensions of each degree, equivalently: $\Gal_K$ has only finitely many open subgroups of each degree.
In particular if $\Gal_K$ is topologically finitely generated then $K$ is bounded.
Perfect bounded $\pac$ fields are model complete by constants \cite{wheeler-pac}.
Pseudofinite fields and infinite extensions of finite fields are bounded $\pac$, in either case boundedness follows from the basic theory of finite fields and $\pac$ follows from the Hasse-Weil estimates, see \cite[11.2.3, 20.10.1]{field-arithmetic}.

\meno
We describe another natural family of bounded $\pac$ fields.
For each $e < \upomega$ let $F_e$ be the free profinite group on $e$ generators.
Note that $F_e$ is topologically finitely generated when $e < \upomega$, so $K$ is bounded when $\Gal_K = F_e$.
Suppose that $K$ is finitely generated over its prime subfield.
Equip $\Gal_K$ with the unique Haar probability measure.
If $\upsigma_1,\ldots,\upsigma_e$ are chosen from $\Gal_K$ independently and at random then with probability one the fixed field of $\upsigma_1,\ldots,\upsigma_e$ is a perfect $\pac$ field with absolute Galois group $F_e$, see \cite[Theorem 20.5.1]{field-arithmetic}.

\meno
Bounded pseudo real closed fields are model complete by constants~\cite[Corollary 3.6]{Samaria-2017}.
See \cite{Samaria-2017} and \cite{prestel-prc} for an overview of pseudo real closed fields.
If $L$ is a field and $<$ is an arbitrary field order on $L$ then the \'etale open topology over $L$ refines the $<$-topology, see \cite[Proposition 6.14]{firstpaper}.
An $n$-ordered field is a structure $(K,<_1,\ldots,<_n)$ where each $<_i$ is a field order on $K$.
Van den Dries has shown that the theory of $n$-ordered fields has a model companion $\Sa O_n$~\cite{Lou-thesis}.
Models of $\Sa O_n$ are pseudo real closed and the absolute Galois group of a model of $\Sa O_n$ is a pro-$2$-group generated by $n$ involutions, hence such a field is bounded.
See Prestel~\cite{prestel-prc} for more information.
Suppose $(K,<_1,\ldots,<_n) \models \Sa O_n$.
Then the $<_i$-topologies are distinct and each $<_i$ is definable in the language of rings~\cite[Lemma 3.5]{Samaria-2017}.
There is also a similar theory of pseudo $p$-adically closed fields, and bounded pseudo $p$-adically closed fields are model complete by constants, see \cite[Section 6]{Samaria-2017}.

\meno
We now discuss Frobenius fields.
A profinite group $G$ has the \textbf{embedding property} if whenever there are finite discrete groups $H,H'$ and continuous epimorphisms $f \colon G \to H$, $g \colon H' \to H$, and $h \colon G \to H'$, then there is a continuous epimorphism $f' \colon G \to H'$ such that $f = g \circ f'$.
A \textbf{Frobenius field} is a $\pac$ field whose absolute Galois group has the embedding property, see \cite[Chapter 24]{field-arithmetic}.
Frobenius fields are model-theoretically tame:
they admit quantifier elimination in a reasonable language (see Fact~\ref{fact:frobenius} below) and are $\mathrm{NSOP}_1$~\cite{C-frob}, the latter is a classification-theoretic property of recent interest.
We give two examples.

\meno
The first example is conjectural.
Let $\qsolve$ be the maximal solvable extension of $\Qq$.
It is a well-known open conjecture that $\qsolve$ is $\pac$~\cite[3.3]{open-problems-ample}.
Fried and Haran have shown that if $\qsolve$ is large then the absolute Galois group of $\qsolve$ has the embedding property~\cite[Theorem 1.5, Theorem 3.9]{Fried-Haran}.
Thus if $\qsolve$ is $\pac$ then $\qsolve$ is Frobenius\footnote{In an earlier version of this paper we gave an incorrect justification for conjectural Frobeniusness of $\qsolve$.
Arno Fehm alerted us to this error and made us aware of the work of Fried and Haran.}.

\meno
We now describe an interesting theory of Frobenius fields.
Recall that $K$ is \textbf{$\upomega$-free} if for any Galois extension $L/K$, finite group $G$, and surjective homomorphism $f \colon G \to \Gal(L/K)$ there is an extension $L'/L$ and an isomorphism $g \colon \Gal(L'/K) \to G$ such that $L/K$ is Galois and $f \circ g$ agrees with the restriction $\Gal(L'/K) \to \Gal(L/K)$.
If $K$ is countable then $K$ is $\upomega$ -free if and only if $\Gal_K = F_\upomega$~\cite[24.8.2]{field-arithmetic}.
Note that $\upomega$-freeness trivially implies the embedding property, hence an $\upomega$ -free field is Frobenius.
Let $\Sa L$ be the expansion of the language of rings by relation symbols $R_2,R_3,\ldots$ where each $R_m$ is $m$-ary.
We consider any field to be an $\Sa L$-structure by declaring
$$ R_m(x_0,\ldots,x_{m - 1}) \quad\Longleftrightarrow\quad \exists t (t^m + x_{m - 1} t^{m - 1} + \cdots + x_2 t^2 + x_1 t + x_0 = 0) \quad \text{for all  } m \ge 2 . $$
Note that a field extension $L/K$ induces an $\Sa L$-embedding if and only if $K$ is relatively algebraically closed in $L$.
The $\Sa L$-theory of fields has a model companion.
A characteristic zero field is existentially closed as an $\Sa L$-structure if and only if $K$ is $\pac$ and $\upomega$ -free~\cite[27.2.3]{field-arithmetic}.
It follows that any characteristic zero field has a regular extension which is $\pac$ and $\upomega$ -free, hence Frobenius.

\meno
We know very little about general model complete fields.
All known model complete fields are large.
Macintyre has asked if a model complete field is bounded and Koenigsmann has conjectured that a bounded field is large~\cite[p.~496]{JK-slim}. 

\begin{question-star}
\label{conj:model-complete}
Is every model complete field large?
\end{question-star}

Equivalently: is every model complete field \'ez? (The above question has appeared in~\cite[Question 8]{JK-slim} as well.)
We now describe a related conjecture.
Let $\kalg$ be the algebraic closure of $K$.
We say that $K$ has \textbf{almost quantifier elimination} if any formula $\phi(x), x = (x_1,\ldots,x_m)$ is equivalent to a formula $\exists y \theta(x,y)$ where $y = (y_1,\ldots,y_n)$, $\theta$ is quantifier free possibly with parameters from $K$, and $\kalg \models \forall x \exists^{\le k} y  \theta(x,y)$ for some $k$.
It is easy to see that $K$ has almost quantifier elimination if and only if every definable subset of $K^m$ is of the form $f(V(K))$ for a quasi-finite morphism $f:V \to \Aa^m$ of $K$-varieties.
Many of the familiar examples of model complete fields have almost quantifier elimination, this includes pseudofinite fields and field which are algebraically, real, or $p$-adically closed.

The following conjecture is due to Pillay.
See~\cite[Chapter 2]{cousins} for related questions.

\begin{conjecture-star}[Pillay]
If $K$ has almost quantifier elimination then $K$ is large.
\end{conjecture-star}

Equivalently: a field with almost quantifier elimination is \'ez.

\subsection*{How we prove Theorem A}

The details of the proof appears in Section~\ref{section:consequences}. The proof is a straightforward application of Theorem F and Noetherian induction.

\begin{thmF}
\label{thm:open}
Suppose that $K$ is perfect and $V \to W$ is dominant morphism between irreducible $K$-varieties.
Then there is a dense open subvariety $U$ of $V$ such that $U(K) \to W(K)$ is $\Sa E_K$-open.
\end{thmF}

Theorem F is also crucial for the proof of Theorem E.

\meno
The characteristic zero case of Theorem F is a consequence of generic smoothness of dominant morphisms in characteristic zero (algebraic Sard's theorem).
Generic smoothness fails in positive characteristic, in this case we factor $V \to W$ as $V \to V' \to W$ where $V \to V'$ is a universal homeomorphism and the field extension $K(V')/K(W)$ induced by $V' \to W$ is separable, hence $V' \to W$ is generically smooth.
This decomposition arises from a decomposition of the function field extension $K(V)/K(W)$ into a purely inseparable extension and a separable extension.
The key lemma is that if $K$ is perfect then a universal homeomorphism $V \to W$ of $K$-varieties induces an $\Sa E_K$-homeomorphism $V(K) \to W(K)$.

\subsection*{Acknowledgements}
We thank Will Johnson and Chieu-Minh Tran for very useful conversations.
The word ``\'ez" is due to Minh and is pronounced ``easy".
The proof of Theorem A owes a debt to Arno Fehm: our original proof of Theorem A made crucial use of ideas from Fehm's proof of Fact~\ref{fact:fehm} below.
Ye was partially supported by GeoMod AAPG2019 (ANR-DFG), Geometric and Combinatorial Configurations in Model Theory.

\section{Conventions and background}

\subsection{Basic conventions}
Throughout $m,n,i,j,k,r$ are natural numbers.
Given a tuple $a = (a_1,\ldots,a_n)$ we let $a^k = (a_1^k,\ldots,a^k_n)$.
A ``$K$-variety" is a separated reduced $K$-scheme of finite type.
By ``morphism" without modification we mean a morphism of $K$-varieties.
Let $V$ be a $K$-variety.
We let $\dim V$ be the usual algebraic dimension of $V$ and if $X$ is an arbitrary subset of $V$ then we let $\dim X$ be the dimension of the Zariski closure of $X$.
A \textbf{subvariety} of $V$ is an open subvariety of a closed subvariety of $V$.
A subset $X$ of $V$ is \textbf{constructible} if it is a finite union of subvarieties of $V$, equivalently if it is a boolean combination of closed subvarieties of $V$.
We let $V(K)$ be the set of $K$-points of $V$, $K[V]$ be the coordinate ring of $V$, and $K(V)$ be the function field of $V$ when $V$ is irreducible. For an ideal $I$, we use $\mathrm{rad}(I)$ to denote its radical.
We let $\Aa^m$ be $m$-dimensional affine space over $K$, i.e. $\Aa^m = \Spec K[x_1,\ldots,x_m]$.
Recall that $\Aa^m(K) = K^m$.

\meno
Suppose that $W$ is a scheme.
A \textbf{$W$-scheme} is a scheme $V$ equipped with a morphism $V \to W$.
Given $W$-schemes $V \to W$ and $V' \to W$ a morphism $V \to V'$ of $W$-schemes is a morphism of schemes such that the diagram below commutes.

\begin{center}
\begin{tikzcd}
V \arrow[rr] \arrow[rd] &   & V' \arrow[ld] \\
                        & W &              
\end{tikzcd}
\end{center}
Note that $W$-schemes and $W$-scheme morphisms form a category.
The category of \textbf{\'etale schemes over $W$} is the full subcategory of $W$-schemes $V$ such that $V\to W$ is \'etale.
If $W$ is a $K$-variety, and $V$ is an \'etale $W$-scheme, then $V$ is again a $K$-variety.

\meno
All facts below are presumably unoriginal.
We include proofs for the sake of completeness.

\begin{fact}
\label{fact:finite-var}
Suppose $V$ is $K$-variety.
Then $|V| < \infty$ if and only if $\dim V = 0$.
\end{fact}

\begin{proof}
Suppose $\dim V \ge 1$.
Note that $V$ contains an open subvariety of the form $\Spec A$ for a finitely generated $K$-algebra $A$ of dimension $\dim V$.
By Noether normalization $A$ is an integral extension of a polynomial ring over $K$ and hence has infinitely many points.
Suppose $\dim V = 0$.
It is enough to show that every affine open subset of $V$ has finitely many points.
Suppose $\Spec A$ is an affine open subset of $V$.
Then $A$ is an Artinian $K$-algebra, hence finite.
In particular $\Spec A$ has finitely many points.
\end{proof}

\begin{fact}
\label{fact:constructible}
Suppose that $V,W$ are $K$-varieties and $X,Y \subseteq V$ are constructible.
\begin{enumerate}[leftmargin=*]
\item If $X$ is Zariski dense in $Y$ then $X$ contains a dense open subvariety of $Y$ and\\ $\dim Y \setminus X < \dim Y$.
\item If $\overline{X}$ is the Zariski closure of $X$ in $V$ then $\dim \overline{X} \setminus X < \dim X$.
\end{enumerate}
Suppose that $f \colon V \to W$ is a morphism.
Then
\begin{enumerate}[leftmargin=*]
\setcounter{enumi}{2}  
\item $Z=\{a \in W : |f^{-1}(a)| < \infty \}$ is Zariski open.
Moreover, there is $n$ such that $|f^{-1}(a)| \le n$ for all $a \in Z$.
\item $f(X)$ is a constructible subset of $W$ and $\dim f(X) \le \dim X$.
\item If $|f^{-1}(a)| < \infty$ for all $a \in W$ then $\dim f(V) = \dim V \le \dim W$.
\end{enumerate}
\end{fact}

We let $\kappa(a)$ be the residue field of $a \in W$.

\begin{proof}
(1) follows by \cite[Lemma 005K]{stacks-project}, $(2)$ is a special case of $(1)$.
We describe a proof of $(3)$.
We let $V_a$ be the scheme-theoretic fiber of $V$ over $a \in W$.
The underlying set of each $V_a$ is $f^{-1}(a)$.
By ~\cite[Theorem 13.1.3]{EGA-IV-4} $Z := \{a\in W: \dim V_a =0\}$ is Zariski open.
Note that each $V_a$ is a $\kappa(a)$-variety and apply Fact~\ref{fact:finite-var}.
We now produce $n$.
After replacing $W$ with $X$ and $V$ with $f^{-1}(X)$, we may assume that $f$ is quasi-finite.
By Zariski's main theorem there is a $K$-variety $V'$,  an open immersion $i \colon V \to V'$, and a finite morphism $g\colon  V'\to W$ such that $f = g\circ i$.
Let $n$ be the degree of $g$.
Then $|g^{-1}(a)| \le n$ for all $a \in W$, so $|f^{-1}(a)|\le n$ for all $a \in W$.
The first claim of (4) is a special case of Chevalley's theorem on constructible sets.
We prove the second claim.
After replacing $V$, $W$ with the Zariski closure of $X$, $f(X)$, respectively, we suppose that $X$ is Zariski dense in $V$ and $f(X)$ is Zariski dense in $W$.
Then $\dim X = V$ and $\dim f(X) = \dim W$.
By (1) $f(X)$ contains a dense open subvariety of $W$.
Thus $V \to W$ is dominant so $\dim W \le \dim V$.
For (5), by Zariski's main theorem, it suffices to show this when $f$ is a finite morphism. This follows from~\cite[Lemma 0ECG]{stacks-project}.
\end{proof}

\begin{fact}
\label{lem:dim-fin-union}
Suppose that $V$ is a $K$-variety, $X$ is a subset of $V$, and $X = X_1\cup\cdots\cup X_k$.
Then $\dim X$ is the maximum of $ \dim X_1,\ldots, \dim X_k $.
\end{fact}

We let $\overline{X}$ be the Zariski closure of $X$ in $V$.
Note that $\dim Y = \dim \overline{Y}$ holds for any $Y \subseteq V$.

\begin{proof}
We have $\overline{X} = \bigcup_{i=1}^{k} \overline{X_i}$, so we may suppose each $X_i$ is Zariski closed. 
The fact now follows from the definition of the dimension of a Noetherian space.
\end{proof}

\begin{fact}
\label{fact:decomp}
Suppose that $K$ is perfect, $V$ is a $K$-variety, and $V_1,\ldots,V_k$ are closed subvarieties of $V$ such that $V = \bigcup_{i = 1}^{k} V_i$.
Then there are pairwise disjoint smooth irreducible subvarieties $W_1,\ldots,W_\ell$ of $V$ such that $V = \bigcup_{i = 1}^{\ell} W_i$ and each $W_j$ is either contained in or disjoint from every $V_i$.
\end{fact}

\begin{proof}
For each $I \subseteq \{1,\ldots,k\}$ we let $V_I = \left(\bigcap_{i \in I} V_i \right) \setminus \left( \bigcup_{i \notin I} V_i \right)$.
Note that each $V_I$ is a subvariety of $V$, the $V_I$ are pairwise disjoint, and $V = \bigcup_{I \subseteq \{1,\ldots,k\} } V_I$.
It suffices to fix $I$ such that $V_I$ is nonempty and show that $V_I$ is a union of a finite collection of pairwise disjoint smooth irreducible subvarieties.
Thus we may suppose that $k = 1$ and $V_1 = V$.

\meno
We now apply induction on $\dim V$.
If $\dim V = 0$ then $V$ is finite and we let $W_1,\ldots,W_{\ell}$ be the irreducible components of $V$.
Suppose $\dim V \ge 1$.
The irreducible components of a smooth variety are pairwise disjoint, so it is enough to produce pairwise disjoint smooth subvarieties $W_1,\ldots,W_{\ell}$ of $V$ such that $V = \bigcup_{i = 1}^{\ell} W_i$.
Let $W_1$ be the smooth locus of $V$.
As $K$ is perfect $W_1$ agrees with the regular locus of $V$ \cite[Proposition 3.5.22]{poonen-qpoints}, which is open by~\cite[(29.E) Remark 1]{matsumura}.
The generic point of any irreducible component of $V$ is regular, so $\dim V \setminus W_1 < \dim V$.
By induction there are pairwise disjoint smooth subvarieties $W_2,\ldots,W_{\ell}$ of $V \setminus W_1$ such that $V \setminus W_1 =  W_2\cup\cdots\cup W_\ell$.
\end{proof}

Finally, we leave the easy proof of Fact~\ref{fact:sub-frontier} to the reader.

\begin{fact}
\label{fact:sub-frontier}
Suppose that $V$ is a $K$-variety, $W$ is a subvariety of $V$, and $\overline{W}$ is the Zariski closure of $W$ in $V$.
Then $\overline{W} \setminus W$ is a closed subvariety of $V$.
\end{fact}

Fact~\ref{fact:one-poly} is certainly well-known, but we do not know a reference.

\begin{fact}
\label{fact:one-poly}
Suppose that $K$ is not algebraically closed and $V$ is a closed subvariety of $\Aa^m$.
Then there is $f \in K[x_1,\ldots,x_m]$ such that $V(K) = \{\alpha \in K^m : f(\alpha) = 0 \}$.
\end{fact}

Given $f \in K[x_1,\ldots,x_m]$ we let $Z(f)$ be $\{ \alpha \in K^m : f(\alpha) = 0\}$.

\begin{proof}
As $K[x_1,...,x_m]$ is Noetherian there are $g_1,\ldots,g_n \in K[x_1,\ldots,x_m]$ such that we have $V=\Spec K[x_1,\ldots,x_m]/(g_1,...,g_n)$.
Then $V(K) = \bigcap_{i = 1}^{n} Z(g_i)$.
Therefore it is enough to fix $g,h \in K[x_1,\ldots,x_m]$ and produce $f \in K[x_1,\ldots,x_m]$ such that $Z(f) = Z(g) \cap Z(h)$.
Let $p \in K[t]$ be an irreducible polynomial of degree $\ge 2$ and $q(t,t')$ be the homogenization of $p(t)$.
If $q(\alpha,\beta) = 0$ for some $\alpha,\beta \in K$, then $\alpha = 0 = \beta$.
Take $f=q(g,h)$. 
\end{proof}

\subsection{The relative Frobenius}
\label{section:relative-frob}
We recall backgroud on the relative Frobenius.
Our reference is SGA $5$ \cite[Expose XV]{SGA5}.
\textbf{We suppose that $\Chara(K) = p > 0$ and $V \to W$ is a dominant morphism of irreducible $K$-varieties}.
We first prove an elementary field-theoretic lemma to be applied to the function field of $V$.

\begin{lemma}
\label{lem:decompose}
Suppose that $K$ is perfect, $K(s_1,\ldots,s_m,t_1,\ldots,t_n)$ is a finitely generated extension of $K$, and $s = (s_1,\ldots,s_m), t = (t_1,\ldots,t_n)$.
Then $K(s,t^{p^r})/K(s)$ is separable when $r \ge 1$ is sufficiently large. 
\end{lemma}

\begin{proof}
Let $K(s) \subseteq L_0 \subseteq L_1 \subseteq K(s,t)$ be field extensions such that $L_0/K(s)$ is purely transcendental, $K(s,t)/L_0$ is algebraic, $L_1/L_0$ is separable, and $K(s,t)/L_1$ is purely inseparable.
Then for each $i \in \{1,\ldots,n\}$ there is $r_i$ such that $t^{p^{r_i}}_i \in L_1$.
Let $r = \max\{r_1,\ldots,r_n\}$.
Then $t^{p^r}_i \in L_1$ for all $i$, so $K(s,t^{p^r})$ is contained in $L_1$.
Thus $K(s,t^{p^r})/K(s)$ is separable.
\end{proof}

For a $K$-variety $X$, we let $\fr_X \colon X \to X$ be the absolute Frobenius morphism. This morphism is the identity on the underlying topological space of $X$ and raises every section to the $p$th power.
If $X = \Spec A$ is affine then $\fr_X$ is dual to the Frobenius $A \to A$.
The absolute Frobenius is a morphism of $K$-varieties if and only if $K$ is the field with $p$ elements.
We let $V^{(p)} \to W$ be the pullback of $V \to W$ via $\fr_W$.
Let $\uppi \colon V^{(p)} \to V$ be the projection, so the following diagram is a pullback square.
\[ \xymatrix{
V^{(p)} \ar[r]^{\uppi} \ar[d] & V \ar[d] \\ W
\ar[r]^{\fr_W} & W}\]
We let $\frrel \colon V \to V^{(p)}$ be the \textbf{relative Frobenius of $V$ over $W$}.
This is the morphism induced by the universal property of the pullback square above.
In particular the diagram below commutes.
\[\xymatrix
   { & V^{(p)}  \ar[dr]^{\uppi} & \\
     V \ar[ur]^{\frrel} \ar[rr]_{\fr_V} & & V
   }
\]
The relative Frobenius is a morphism of $W$-schemes, so $V \to W$ factors as
$$V \xrightarrow{\frrel} V^{(p)} \to W.$$
Fact~\ref{fact:homeo} is \cite[\href{https://stacks.math.columbia.edu/tag/0CCB}{Lemma 0CCB}]{stacks-project}.

\begin{fact}
\label{fact:homeo}
$\frrel$ is a homeomorphism $V \to V^{(p)}$.
\end{fact}

Given a $W$-scheme $Y \to W$ and a $W$-scheme morphism $f \colon Y \to V$ we let $f^{(p)} \colon Y^{(p)} \to V^{(p)}$ be the morphism given by base-changing along $\fr_W$.

\meno
We explain the situation in the affine case.
Suppose that $W = \Spec A$ and $V = \Spec B$ for $K$-algebras $A,B$.
Then $V^{(p)} = \Spec B\! \otimes_A \!A$ where the map $A \to A$ is the Frobenius and $\frrel \colon V \to V^{(p)}$ is dual to the map $B \!\otimes_A\! A \to B$ given by $b \otimes a \mapsto b^p a$.

\meno
We also require the $r$-fold iterates of the relative Frobenius.
For all $r \ge 1$ we define $V^{(p^{r+1})}$ to be $(V^{(p^r)})^{(p)}$ and let $\fr^{(r)}_{V/W} \colon V \to V^{(p^r)}$ be given by 
$$\fr^{(r + 1)}_{V/W} = \fr_{V^{(p^r)}/W} \circ \fr^{(r)}_{V/W}.$$
Then $\fr^{(r)}_{V/W}$ is the $r$th iterate of the relative Frobenius.
Furthermore $V \to W$ factors as
$$ V \xrightarrow{\fr^{(r)}_{V/W}} V^{(p^r)} \to W $$
for each $r \ge 1$.
By Fact~\ref{fact:homeo} and induction each $\fr^{(r)}_{V/W}$ is dominant. However, $V^{(p^r)}$ is not reduced in general, see Example~\ref{eg:non-red}. For our purpose, a slightly modified version of the above will be needed to stay in the realm of varieties. We equip $V^{(p^r)}$ with the canonical reduced structure given by the closed immersion $V^{(p^r)}_{\mathrm{red}}\to V^{(p^r)}$ defined as in~\cite[Definition 01J4]{stacks-project}. Since $V$ is reduced, $\fr^{(r)}_{V/W}$ factors through the above as
$$ V \xrightarrow{\fr^{(r),\mathrm{\red}}_{V/W}} V^{(p^r)}_{\mathrm{red}}\to V^{(p^r)} $$
for each $r\geq 1$~\cite[Lemma 0356]{stacks-project}. We call $\fr^{(r),\mathrm{\red}}_{V/W}$ the ($r$-th iterate of) reduced relative Frobenius. In particular, $V^{(p^r)}_{\mathrm{red}}$ is an irreducible variety, so the field extension $K(V)/K(W)$ decomposes into $K(V)/K(V^{(p^r)}_{\mathrm{red}})$ and $K(V^{(p^r)}_{\mathrm{red}})/K(W)$ for all $r \ge 1$.

Fact~\ref{fact:frob-affine} follows by the comments on the affine case above and induction.

\begin{fact}
\label{fact:frob-affine}
If $V$  and $W$ are affine then $V^{(p^r)}$ and $V^{(p^r)}_{\mathrm{red}}$ are affine for all $r \ge 1$.
\end{fact}
\meno
We now make some further remarks on the affine case.
As $V \to W$ is dominant the dual $K$-algebra morphism $K[W] \to K[V]$ is injective, so we consider $K[W]$ to be a subring of $K[V]$.
Let $s = (s_1,\ldots,s_m)$ and $t = (t_1,\ldots,t_n)$ be such that $K[W] = K[s]$ and $K[V] = K[s,t]$.
Let $K[s,y]$ be the polynomial ring over $K[s]$ in the variables $y= (y_1,\ldots,y_m)$.
Let $\uprho \colon K[s,y] \to K[s,t]$ be the $K[s]$-algebra morphism given by $\uprho(y_i) = t_i$ for each $i$, $I$ be the kernel of $\uprho$, and identify $K[s,t]$ with $K[s,y]/I$.

\meno
Given $j = (j_1,\ldots,j_n) \in \Nn^n$ we let $y^j = (y^{j_1}_1,\ldots,y^{j_n}_n)$.
For any $f \in K[s,y]$ we have $f = f_1 y^{j_1} + \cdots + f_k y^{j_k}$
for some $f_1,\ldots,f_k \in K[s]$, and $j_1,\ldots,j_k \in \Nn^n$.
We then let $f^{(p)} = f^p_1 y^{j_1} + \cdots + f^p_k y^{j_k}$ and let $I^{(p)}$ be the ideal generated by $f^{(p)}$, $f$ ranging over $I$.

Then $V^{(p)} = \Spec K[s,y]/I^{(p)}$ and $V^{(p)}_{\mathrm{red}}=\Spec K[s,y]/\mathrm{rad}(I^{(p)})$, where the canonical closed immersion is the dual of the natural quotient map.

\meno

Let $\uptau \colon K[s,y] \to K[s,t]$ be the $K[s]$-algebra morphism given by $\uptau(a) = a$ for all $a \in K[s]$ and $\uptau(y_i) = t^p_i$ for each $i$. Then $\mathrm{rad}(I^{(p)})$ is the kernel of $\uptau$.
Therefore $\uptau$ factors as 
$$K[s,y] \to K[s, y]/\mathrm{rad}(I^{(p)}) \xrightarrow[]{\upsigma} K[s,t]$$ for some injective $K[s]$-algebra morphism $\upsigma$. The reduced relative Frobenius $\frrel^{\mathrm{red}}$ is dual to $\upsigma$.
The image of $\uptau$ is $K[s,t^p]$, so $\upsigma$ gives a $K[s]$-algebra isomorphism $K[s, y]/\mathrm{rad}(I^{(p)})=K[V^{(p)}_\mathrm{red}] \to K[s,t^p]$.
Fact~\ref{fact:iterate} follows by induction.

\begin{fact}
\label{fact:iterate}
Let $V$ and $W$ be affine and $s = (s_1,\ldots,s_m)$, $t = (s_1,\ldots,t_n)$ be such that $K[W] = K[s]$ and $K[V] = K[s,t]$.
For each $r \ge 1$ there is a $K[s]$-algebra isomorphism $K[V^{(p^r)}_{\mathrm{red}}] \to K[s,t^{p^r}]$ and a $K(s)$-algebra isomorphism $K(V^{(p^r)}_{\mathrm{red}}) \to K(s,t^{p^r})$ for each $r \ge 1$.
\end{fact}

We give an example to further illustrate the picture before continuing.

\begin{example}\label{eg:non-red}
Let $V=\Spec K[X,Y]/(Y-X^p)$ and $W=\Spec K[Y]$ where $V\to W$ is the projection map, which is the dual of the inclusion $K[Y]\to K[X,Y]/(Y-X^p)$. Then $V^{(p)}$ is $\Spec K[X,Y]/(Y^p-X^p)$, which is a non-reduced scheme. $\fr_{V/W}:V\to V^{(p)}$ is the dual of the map 
\[
 K[X,Y]/(Y^p-X^p)\to K[X,Y]/(Y-X^p),
 Y\mapsto Y, X\mapsto X^p.
\]
$V^{(p)}_{\mathrm{red}}$ is $\Spec K[X,Y]/(Y-X)$ and the canonical closed immersion $V^{(p)}_{\mathrm{red}}\to V^{(p)}$ is the dual of 
\[K[X,Y]/(Y^p-X^p)\to K[X,Y]/(Y-X), X\mapsto X, Y\mapsto Y.\] $\fr^{\mathrm{red}}_{V/W}:V\to V^{(p)}_{\mathrm{red}}$ is the dual of the map 
\[K[X,Y]/(Y-X)\to K[X,Y]/(Y-X^p),  Y\mapsto Y, X\mapsto X^p.\]
\end{example}

\begin{lemma}
\label{lem:function-field-frob}
$K(V^{(p^r)}_\mathrm{red})/K(W)$ is separable when $r \ge 1$ is sufficiently large.
\end{lemma}

\begin{proof}
The case when $V$ and $W$ are affine follows from Fact~\ref{fact:iterate} and Lemma~\ref{lem:decompose}.
We now reduce the general case to the case when $V$ and $W$ are affine.
Suppose that $U$ is a dense affine open subvariety of $W$ and $O$ is a dense affine open subvariety of $V$ contained in the preimage of $U$.
We have $K(U) = K(W)$, $K(O) = K(V)$, and we identify the extension $K(V)/K(W)$ with $K(O)/K(U)$.
Let $h \colon O \to V$ be the inclusion.
Then $h^{(p)} \colon O^{(p)} \hookrightarrow V^{(p)}$ is an open immersion as open immersions are closed under base change. 
By induction there is an open immersion $O^{(p^r)} \to V^{(p^r)}$ for each $r \ge 1$. By the fact $X_\mathrm{red}\to X$ is functorial, 
we may consider $O^{(p^r)}_{\mathrm{red}}$ to be an open subvariety of $V^{(p^r)}_{\mathrm{red}}$ and identify  $K(V^{(p^r)}_{\mathrm{red}})$ with $K(O^{(p^r)}_{\mathrm{red}})$.
By Fact~\ref{fact:frob-affine} each $O^{(p^r)}_{\mathrm{red}}$ is affine.
The morphism $O \to W$ factors as 
$$O \xrightarrow[]{\fr^{(r),\mathrm{red}}_{O/W}} O^{(p^r)}_{\mathrm{red}} \to W.$$
Thus by Fact~\ref{fact:homeo} the image of $O^{(p^r)}_{\mathrm{red}} \to W$ is contained in $U$.
Therefore the extension $K(V)/K(V^{(p^r)}_{\mathrm{red}})$, $K(V^{(p^r)}_{\mathrm{red}})/K(W)$ can be identified with $K(O)/K(O^{(p^r)}_{\mathrm{red}})$, $K(O^{(p^r)}_{\mathrm{red}})/K(U)$, respectively.
After replacing $V$ with $O$ and $W$ with $U$ we can suppose that both $V$ and $W$ are affine $K$-varieties.
\end{proof}

\subsection{The \'etale open topology}
\label{section:general}
Let $V$ be a $K$-variety.
An \textbf{\'etale image} in $V(K)$ is the image of $X(K) \to V(K)$ for some \'etale morphism $X \to V$ of $K$-varieties.
It is shown in \cite[Theorem A]{firstpaper} that \'etale images in $V(K)$ form a basis for a topology on $V(K)$ refining the Zariski topology which we refer to as the \textbf{\'etale open ($\cE_K$-)topology}.
Fact~\ref{fact:large-potato} is proven in \cite[Theorem C]{firstpaper}.

\begin{fact}
\label{fact:large-potato}
The following are equivalent:
\begin{enumerate}
\item $K$ is large,
\item the \'etale open topology on $K = \Aa^1(K)$ is not discrete,
\item the \'etale open topology on $V(K)$ is non-discrete whenever $V(K)$ is infinite.
\end{enumerate}
\end{fact}

Fact~\ref{fact:first-paper-system} is also proven in \cite{firstpaper}. (1)-(4) follows from~\cite[Proposition 5.5]{firstpaper}, (5) is~\cite[Proposition 7.24]{firstpaper} and (6) is~\cite[Proposition 4.6]{firstpaper}.

\begin{fact}
\label{fact:first-paper-system}
Suppose that $V \to W$ is a morphism between $K$-varieties.
Equip $V(K)$ and $W(K)$ with their \'etale open topologies and let  $V(K) \to W(K)$ be the induced map.
Then:
\begin{enumerate}
\item $V(K) \to W(K)$ is continuous,
\item if $V \to W$ is a (scheme-theoretic) closed immersion then $V(K) \to W(K)$ is a\newline (topological) closed embedding,
\item if $V \to W$ is a (scheme-theoretic) open immersion then $V(K) \to W(K)$ is a\newline (topological) open embedding,
\item if $V \to W$ is \'etale then $V(K) \to W(K)$ is open,
\item the projection $V(K) \times W(K) \to V(K)$ is open when $V(K) \times W(K) = (V \times W)(K)$ is also equipped with the \'etale open topology,
\item the \'etale open topology on $V(K) \times W(K)$ refines the product of the \'etale open topologies on $V(K)$ and $W(K)$.
\end{enumerate}
\end{fact}

Recall that an \'ez subset of $V(K)$ is a finite union of definable \'etale open subsets of Zariski closed subsets of $V(K)$.
For this definition to make sense we need to define the \'etale open topology on a Zariski closed subset of $V(K)$.
If $Z \subseteq V(K)$ is Zariski closed then there is a closed subvariety $W$ of $V$ such that $Z  = W(K)$, so we define the \'etale open topology on $Z$ to agree with the \'etale open topology $W(K)$.
Proposition~\ref{prop:well-define} ensures that this does not depend on choice of $W$.
Proposition~\ref{prop:well-define} follows immediately from the second and third items of Fact~\ref{fact:first-paper-system} and will be used implicitly below at many points.

\begin{proposition}
\label{prop:well-define}
Suppose that $W$ is a subvariety of $V$.
Then the \'etale open topology on $W(K)$ agrees with the subspace topology on $W(K)$ induced by the \'etale open topology on $V(K)$.
If $W'$ is another subvariety of $V$ with $W'(K) = W(K)$ then the \'etale open topology on $W(K)$ agrees with the \'etale open topology on $W'(K)$.
\end{proposition}

Pop has shown that if $K$ is large and $V$ is a smooth irreducible $K$-variety with $V(K) \ne \emptyset$ then $V(K)$ is Zariski dense in $V$~\cite{pop-embedding}.
Fact~\ref{fact:with-anand} generalizes this, it is \cite[Lemma 2.6]{with-anand}.

\begin{fact}
\label{fact:with-anand}
Suppose that $K$ is large and $V$ is a smooth irreducible $K$-variety.
Then any nonempty \'etale open subset of $V(K)$ is Zariski dense in $V$.
\end{fact}

Finally~\ref{fact:hd} is also proven in \cite[Theorem B and C]{firstpaper}.

\begin{fact}
\label{fact:hd}\hspace{.1cm}
\begin{enumerate}[leftmargin=*]
\item If $K$ is separably closed then the $\Sa E_K$-topology on $V(K)$ agrees with the Zariski topology.
\item If $K$ is t-Henselian and not separably closed then the \'etale open topology on $V(K)$ agrees with the t-Henselian topology. 
In particular if $K$ is real closed then the \'etale open topology on $V(K)$ agrees with the order topology and if $K$ is Henselian and not separably closed then the \'etale open topology on $V(K)$ agrees with the valuation topology.
\item If $K$ is not separably closed then the \'etale open topology on $V(K)$ is Hausdorff when $V$ is quasi-projective.
\end{enumerate}
\end{fact}

Implicit in the second statement are the well-known facts that a real closed field admits a unique field order and any two non-trivial Henselian valuations on a non-separably closed field induce the same topology.

\section{Universal homeomorphisms and Galois actions}
\label{section:uni-homeo}
We prove some results on universal homeomorphisms between $K$-varieties. 
We also discuss the action of the automorphism group of $K$.
In this section, and this section only, we work with scheme morphisms between $K$-varieties which are not morphism of $K$-varietiess. 
A morphism $V \to W$ of schemes is a \textbf{universal homeomorphism} if for every $W$-scheme $X$, the morphism $V \times_W X \to X$ produced from $V \to W$ by base change is a homeomorphism, see \cite[\S 2.4.2]{EGAIV-2}.
It is clear from this definition that the collection of universal homeomorphisms is closed under compositions and base change.
In characteristic zero a universal homeomorphism is an isomorphism.
See \cite[Lemma 04DF, Theorem 04DZ]{stacks-project} for Fact~\ref{fact:uni-homeo}.
As before we let $\kappa(\alpha)$ be the residue field of point $\alpha$ on a scheme.

\begin{fact}
\label{fact:uni-homeo}
Let $V,W$ be schemes and $f\colon V \to W$ be a universal homeomorphism.
Then:
\begin{enumerate}
\item $f$ is integral, universally injective, and universally surjective.
\item If $f(\alpha) = \beta$ then the induced field extension $\kappa(\alpha)/\kappa(\beta)$ is purely inseparable.
\item  The functor $X \mapsto X_V =X\times_W V$ is an equivalence of categories between the category of \'etale schemes over $W$ and the category of \'etale schemes over $V$.
\end{enumerate}
\end{fact}
 
Lemma~\ref{lem:uni-homeo} is well-known, we include a proof for the sake of completeness.

\begin{lemma}
\label{lem:uni-homeo}
Suppose $K$ is perfect, $V$ and $W$ are $K$-varieties, and $f \colon V \to W$ is a $K$-variety morphism and a universal homeomorphism.
The induced map $V(K) \to W(K)$ is a bijection.
\end{lemma}

\begin{proof}
Note that $f$ is bijective as $f$ is a homeomorphism.
Therefore $V(K) \to W(K)$ is injective.
We show that $V(K) \to W(K)$ is surjective.
Fix $\beta \in W(K)$.
As $f$ is surjective there is $\alpha \in V$ such that $f(\alpha) = \beta$.
Let $\kappa(\alpha)/\kappa(\beta)$ be the induced field extension and note that $\kappa(\beta) = K$.
By Fact~\ref{fact:uni-homeo}.1 $f$ is integral, hence $\kappa(\alpha)/K$ is algebraic.
By Fact~\ref{fact:uni-homeo}.2 $\kappa(\alpha)/K$ is purely inseparable, so $\kappa(\alpha) = K$ as $K$ is perfect.
Therefore $\alpha \in V(K)$.
\end{proof}

\begin{proposition}
\label{prop:uni-homeo}
Suppose that $K$ is perfect, $V$ and $W$ are $K$-varieties, and $f \colon V \to W$ is a morphism of $K$-varieties and a universal homeomorphism.
Then the map $V(K) \to W(K)$ induced by $f$ is an $\Sa E_K$-homeomorphism.
\end{proposition}

\begin{proof}
By Lemma~\ref{lem:uni-homeo} $V(K) \to W(K)$ is a bijection.
By Fact~\ref{fact:first-paper-system}.1 $V(K) \to W(K)$ is $\Sa E_K$-continuous.
We show that $V(K) \to W(K)$ is $\Sa E_K$-open.
Let $X$ be a $K$-variety and $g\colon X\to V$ be \'etale.
It is enough to show that $f(g(X(K))$ is \'etale open.
By Fact~\ref{fact:uni-homeo}.3 there is an \'etale morphism $h\colon Y\to W$ such that $g\colon X\to V$ is the base change of $h$ along $f$. Taking $K$-points, we have the following pullback square.
 \begin{equation*}
      \xymatrix{ X(K) \ar[r]^{f_h} \ar[d]^{g} & Y(K) \ar[d]^{h} \\ V(K) \ar[r]^{f} & W(K)}
  \end{equation*}
Note that both $f$ and $f_h$ are bijections. Hence  $f(g(X(K))=h(Y(K))$, which is \'etale open.
\end{proof}
Next we look at Galois actions.
Let $\upsigma\colon K \to K$ be an automorphism, we also use $\upsigma$ to denote the map $\upsigma\colon K^n\to K^n$ $(c_1,...,c_n)\mapsto (\upsigma(c_1),...,\upsigma(c_n))$.
We have the following:

\begin{proposition}\label{prop:galois-action}
$\upsigma\colon K^n\to K^n$ as defined above is a homeomorphism with respect to $\cE_K$.
\end{proposition}
\begin{proof}
The map $\upsigma\colon K^n\to K^n$ can be seen as the induced by the dual of the following isomorphism of rings (abusing notation, it is still denoted by $\upsigma$):
\[
\upsigma:K[x_1,...,x_n]\to K[x_1,...,x_n] : x_i\mapsto x_i \text{  }c \mapsto \upsigma(c) \text { for }c \in K
\]
We use $\upsigma^*$ to denote the induced scheme morphism $\Aa^n \to \Aa^n$.
Note that $\upsigma^*$ is invertible.
It therefore suffices to show that $\upsigma \colon K^n \to K^n$ is $\Sa E_K$-open.
Let $e\colon  U\to \mathbb{A}_K^n$ be an \'etale morphism of $K$-varieties.
We have $e^\upsigma\colon U^\upsigma \to \mathbb{A}_K^n$ such that the following is a pullback diagram:
\[
\begin{tikzcd}
U^\upsigma \arrow[dd, "e^\upsigma"] \arrow[rr, " "] &  & U \arrow[dd, "e"] \\
                                   &  &                     \\
\mathbb{A}_K^n\arrow[rr, "\upsigma^*"]                  &  & \mathbb{A}^n_K                 
\end{tikzcd}
\]
Note that ${\upsigma^{*}}(e^\upsigma(U^\upsigma(K))=e(U(K))$ by construction. This finishes the proof.
\end{proof}

\begin{corollary}
\label{cor:aut}
Suppose that $\phi : K \to K$ is an automorphism.
Then $\phi$ is an $\Sa E_K$-homeomorphism.
In particular if $K$ is perfect and $\Chara(K) = p > 0$ then the Frobenius map $K \to K$ given by $a \mapsto a^p$ is an $\Sa E_K$-homeomorphism.
\end{corollary}

It should also be noted that both Proposition~\ref{prop:galois-action} and Corollary~\ref{cor:aut} are more or less obvious as the \'etale open topology is defined in an automorphism-invariant manner.

\begin{corollary}
\label{cor:closed-subfield}
Suppose that $K$ is not separably closed and $\phi \colon  K \to K$ is an automorphism of $K$.
Then the fixed field of $\phi$ is an $\Sa E_K$-closed subset of $K$.
If $A$ is a collection of automorphisms of $K$ then the fixed field of $A$ is an $\Sa E_K$-closed subset of $K$.
\end{corollary}

The second claim of Corollary~\ref{cor:closed-subfield} follows directly from the first.
The first follows from Corollary~\ref{cor:aut}, Fact~\ref{fact:hd}, and the elementary fact that if $T$ is a Hausdorff topological space and $f \colon  T \to T$ is continuous then the set of fixed points of $f$ is closed.
Corollary~\ref{cor:closed-subfield} fails when $K$ is separably closed, as any infinite proper subfield of $K$ is dense and co-dense in the Zariski topology on $K$.

\meno
Suppose that $L/K$ is a field extension and $V$ is a $K$-variety. Since one can naturally identify $V(L)$ with $V_L(L)$, we wish to equip the set $V(L)$ of $L$-points of $V$ with the $\Sa E_L$-topology. And for an intermediate field, we identify $V(F)\subseteq V(L)$ via the canonical embedding.
A slight technical issue arises as $V_L$ might be a non-reduced $L$-scheme, and hence not an $L$-variety, when $L/K$ is inseparable.
In \cite{firstpaper} this issue was handled by working with the slightly broader class of separated finite type $L$-schemes.
However, at present we only need the case when $L/K$ is separable, and in this case $V_L$ is an $L$-variety~\cite[030U]{stacks-project}. 

\meno
Corollary~\ref{cor:galois-action} follows by relativizing the proof of Proposition~\ref{prop:galois-action} to $\upsigma\in \mathrm{Aut}(L/K)$.

\begin{corollary}
\label{cor:galois-action}\quad
\begin{enumerate}
\item Suppose that $L$ is a field.
If $\phi \colon  L \to L$ is an automorphism with fixed field $K$ and $V$ is a $K$-variety, then the map $V(L) \to V(L)$ induced by $\phi$ is an $\Sa E_L$-homeomorphism.
\item If $G$ is a subgroup of the automorphism group of $L$ with fixed field $K$ and $V$ is a $K$-variety, then the action of $G$ on $V(L)$ is an action by $\Sa E_L$-homeomorphisms.
\end{enumerate}
\end{corollary}
\begin{proof}
(2) follows from (1). For (1), for (embedded) affine $V$, this follows from Proposition~\ref{prop:galois-action} and the fact that $V$ is invariant under $\phi$. The general case follows from gluing.
\end{proof}
\begin{corollary}
\label{cor:galois-action-1}
Suppose that $L/K$ is a Galois field extension, $L$ is not separably closed, and $V$ is a $K$-variety.
If $K \subseteq F \subseteq L$ is a subfield then $V(F)$ is an $\Sa E_L$-closed subset of $V(L)$.
\end{corollary}

\begin{proof}
We first consider the case when $V$ is quasi-projective. In this case the $\cE_K$-topology is Hausdorff on $V(K)$ by the second claim of Fact~\ref{fact:hd}.
So we apply Corollary~\ref{cor:galois-action}, and the fact that the fixed points of a continuous self-map of a Hausdorff topological space form a closed set.
We now treat the case when $V$ is an arbitrary $K$-variety.
Let $U_1,\ldots,U_k$ be $K$-affine open subvarieties of $V$ that cover $V$.
Note that the action of $\Gal(L/K)$ on $V(L)$ preserves each $U_{i}(L)$.
Fix a subfield $K \subseteq F \subseteq L$.
The quasi-projective case shows that $U_{i}(F)$ is an $\Sa E_L$-closed subset of $U_{i}(L)$ for each $i$.
Note that $U_{i}(L) \cap V(F) = U_{i}(F)$ for each $i$.
It follows that $V(F)$ is closed.
\end{proof}

We now recall Fact~\ref{fact:new-points}, proven in Fehm~\cite{Fehm-subfield}.

\begin{fact}
\label{fact:new-points}
Suppose that $L$ is large, $K$ is a proper subfield of $L$, and $V$ is a positive dimensional irreducible $K$-variety with a smooth $K$-point.
Then $|V(L) \setminus V(K)| = |L|$.
\end{fact}

Recall that the class of large fields are closed under algebraic extensions.
Suppose that $K$ is large, $L/K$ is Galois, $L$ is not separably closed, and $V$ is a positive-dimensional irreducible $K$-variety with a smooth $K$-point.
For example if the maximal abelian extension $\Qq_\mathrm{ab}$ of $\Qq$ is large as conjectured then we can take $K = \Qq_\mathrm{ab}$ and $L = \qsolve$.
Corollary~\ref{cor:galois-action-1} shows that $F \mapsto V(F)$ gives a morphism from the lattice of intermediate subfields of $L/K$ to the lattice of $\Sa E_L$-closed subsets of $V(L)$.
Fact~\ref{fact:new-points} shows that this morphism is injective.

\section{Proof of Theorem F}
\label{section:openess}

\subsection{The characteristic zero case}

This case follows from an algebraic analogue of Sard's theorem and Proposition~\ref{prop:smooth-iff-etale}.
Fact~\ref{fact:factor-smooth} is \cite[\textsection 2.2 Proposition 11]{Neron}.

\begin{fact}
\label{fact:factor-smooth}
Suppose that $f \colon  V \to W$ is a smooth morphism of $K$-varieties, $p \in V$, and the relative dimension of $f$ at $p$ is $n \ge 1$.
Then there is an open subvariety $U$ of $V$ containing $p$ such that the restriction of $f$ to $U$ factors as $\uppi \circ g$ for an \'etale morphism $g \colon  U \to W \times \Aa^n$ and the projection $\uppi \colon  W \times \Aa^n \to W$.
\end{fact}
\begin{proposition}
\label{prop:smooth-iff-etale}
Suppose that $f \colon  V \to W$ is a smooth morphism of $K$-varieties.
Then $V(K) \to W(K)$ is $\Sa E_K$-open.
\end{proposition}
\noindent
We now prove Proposition~\ref{prop:smooth-iff-etale}.

\begin{proof}
Fix $p \in V(K)$.
We show that $V(K) \to W(K)$ is open at $p$.
Let $n$ be the relative dimension of $f$ at $p$.
Suppose $n = 0$.
Then $f$ is \'etale at $p$, so $f$ is \'etale on an open subvariety $U$ of $V$ containing $p$.
By Fact~\ref{fact:first-paper-system}.4 the restriction of $f$ to $U(K)$ is $\Sa E_K$-open.
Suppose that $n \ge 1$.
Let $U$, $g \colon  U \to W \times \Aa^n$, and $\uppi \colon  W \times \Aa^n \to W$ be as in Fact~\ref{fact:factor-smooth}.
By Fact~\ref{fact:first-paper-system}.4 $U(K) \to W(K) \times K^n$ is $\Sa E_K$-open and by Fact~\ref{fact:first-paper-system}.5 $W(K) \times K^n \to W(K)$ is $\Sa E_K$-open.
Hence the restriction of $f$ to $U(K)$ is $\Sa E_K$-open.
\end{proof}

Fact~\ref{fact:sard} is an algebraic analogue of Sard's theorem.
See \cite[Corollary 5.4.2]{Mumford-Oda} for a proof.
The statement in \cite{Mumford-Oda} only covers the case when $W$ is regular, but the proof goes through in the more general case verbatim.

\begin{fact}
\label{fact:sard}
Suppose that $V \to W$ is a dominant morphism of irreducible $K$-varieties.\\
The following are equivalent:
\begin{enumerate}
\item the extension $K(V)/K(W)$ of function fields associated to $V \to W$ is separable,
\item there is a dense open subvariety $U$ of $V$ such that $U \to W$ is smooth.
\end{enumerate}
If $\Chara(K) = 0$ then there is a dense open subvariety $U$ of $V$ such that $U \to W$ is smooth.
\end{fact}

Proposition~\ref{prop:open-map-char-0} follows by Proposition~\ref{prop:smooth-iff-etale} and Fact~\ref{fact:sard}.
This gives the characteristic zero case of Theorem F.

\begin{proposition}
\label{prop:open-map-char-0}
Suppose that $V \to W$ is a dominant morphism of irreducible $K$-varieties.
If the field extension $K(V)/K(W)$ associated to $V \to W$ is separable then there is a dense open subvariety $U$ of $V$ such that $U(K) \to W(K)$ is $\Sa E_K$-open.
In particular if $\Chara(K) = 0$ then there is a dense open subvariety $U$ of $V$ such that $U(K) \to W(K)$ is $\Sa E_K$-open.
\end{proposition}

\subsection{The positive characteristic case}
\label{section:model-complete}
We treat the positive characteristic case of Theorem F.
We use the notation of Section~\ref{section:relative-frob}.
Fact~\ref{fact:frob-uni} is~\cite[Lemma 0CCB,054M,04DF]{stacks-project}.

\begin{fact}
\label{fact:frob-uni}
Suppose that $\Chara(K) = p > 0$ and $V \to W$ is a morphism of $K$-varieties.
Then $\frrel^{(r)}\colon V\to V^{(p^r)}$ and $V^{(p^r)}_\mathrm{red}\to V^{(p^r)}$ are universal homeomorphisms for every $r\geq 1$. Hence $\frrel^{(r),\mathrm{red}}\colon V\to V^{(p^r)}_{\mathrm{red}}$ is a universal homeomorphism for all $r\geq 1$. 
\end{fact}
Corollary~\ref{cor:frob-hom} follows from Fact~\ref{fact:frob-uni} and Proposition~\ref{prop:uni-homeo}.
\begin{corollary}\label{cor:frob-hom}
Suppose that $K$ is perfect, $\Chara(K) = p > 0$, $V \to W$ is a dominant morphism of irreducible $K$-varieties, and $r \ge 1$.
Then the map $V(K) \to V^{(p^r)}_{\mathrm{red}}(K)$ induced by $\fr^{(r),\mathrm{red}}_{V/W}$ is an $\Sa E_K$-homeomorphism.
\end{corollary}
We now prove the positive characteristic case of Theorem F.
Suppose that $K$ is perfect, $\Chara(K) = p > 0$, and $V \to W$ is a dominant morphism of $K$-varieties.
Applying Lemma~\ref{lem:function-field-frob} we fix $r \ge 1$ such that $K(V^{(p^r)}_{\mathrm{red}})/K(W)$ is separable.
By Proposition~\ref{prop:open-map-char-0} there is a dense open subvariety $U'$ of $V^{(p^r)}_{\mathrm{red}}$ such that $U'(K) \to W(K)$ is $\Sa E_K$-open.
Let $U = (\fr^{(r),\mathrm{red}}_{V/W})^{-1}(U')$.
By Fact~\ref{fact:homeo} $U$ is a dense open subvariety of $V$.
Factor $U(K) \to W(K)$ as 
$$U(K) \to U'(K) \to W(K).$$
By Corollary~\ref{cor:frob-hom} $ U(K) \to U'(K)$ is an $\Sa E_K$-homeomorphism.
Thus $U(K) \to W(K)$ is $\Sa E_K$-open.

\meno
We now drop the assumption that $\Chara(K) \ne 0$.

\begin{corollary}
\label{cor:perfect}
Suppose that $K$ is perfect and $f \colon  V \to W$ is a dominant morphism of irreducible $K$-varieties with $\dim V = \dim W$.
Then there is a dense open subvariety $U$ of $W$ such that $f^{-1}(U)(K) \to U(K)$ is $\Sa E_K$-open.
\end{corollary}

\begin{proof}
By Theorem F there is a dense open subvariety $U'$ of $V$ such that $U'(K) \to W(K)$ is $\Sa E_K$-open.
We have $\dim V \setminus U' < \dim V$.
By Fact~\ref{fact:constructible}.4 we have $$\dim f(V' \setminus U) \le \dim V' \setminus U < \dim W.$$
Thus there is a dense open subvariety $U$ of $W$ which is disjoint from $f(V \setminus U')$.
Then $f^{-1}(U)$ is contained in $U'$, hence $f^{-1}(U)(K) \to U(K)$ is $\Sa E_K$-open.
\end{proof}

\section{Proofs of Theorems A and B}
\label{section:consequences}

\subsection{\'Ez sets}
\label{section:E-sets}
Suppose that $V$ is a $K$-variety.
A \textbf{basic \'ez set} is a definable \'etale open subset of a Zariski closed subset of $V(K)$.
An \textbf{\'ez set} is a finite union of basic \'ez sets.
We first establish some facts about \'ez sets and in particular show that the collection of \'ez sets is closed under various operations.
Note that any basic \'ez subset of $V(K)$ is of the form $O \cap Y$ where $O$ is an \'etale open subset of $V(K)$, $Y$ is a Zariski closed subset of $V(K)$, and $O \cap Y$ is definable.
We do not know if we can take $O$ to be definable.

\begin{lemma}
\label{lem:decomp}
Suppose that $K$ is perfect, $V$ is a $K$-variety, and $X$ is an \'ez subset of $V(K)$.
Then there are pairwise disjoint smooth irreducible subvarieties $V_1,\ldots,V_k$ of $V$ and $X_1,\ldots,X_k$ such that each $X_i$ is a definable \'etale open subset of $V_i(K)$ and $X =X_1\cup\cdots\cup X_k$.
\end{lemma}

\begin{proof}
Let $W_1,\ldots,W_{\ell}$ be closed subvarieties of $V$ and $Y_1,\ldots,Y_{\ell}$ be such that each $Y_i$ is a definable \'etale open subset of $W_i(K)$ and $X = \bigcup_{i = 1}^{\ell} Y_i$.
After possibly replacing $V$ with $\bigcup_{i = 1}^{\ell} W_i$ we suppose that the $W_i$ cover $V$.
Applying Fact~\ref{fact:decomp} we obtain pairwise disjoint smooth irreducible subvarieties $V_1,\ldots,V_k$ of $V$ such that $V = \bigcup_{i = 1}^{k} V_i$ and each $V_i$ is either contained in or disjoint from every $W_j$.
For each $i \in \{1,\ldots,k\}$ let $X_i = \bigcup_{j = 1}^{\ell} ( V_i(K) \cap Y_j )$.
Note that if $V_i$ is contained in $W_j$ then $V_i(K) \cap Y_j$ is an \'etale open subset of $V_i(K)$, hence each $X_i$ is an \'etale open subset of $V_i(K)$.
Finally note that each $X_i$ is definable.
\end{proof}

\begin{lemma}
\label{lem:ez-equiv}
Let $V$ be a $K$-variety and $X$ be a subset of $V(K)$.
Then the following are equivalent:
\begin{enumerate}
[leftmargin=*]
\item $X$ is \'ez,
\item $X$ is definable and a finite union of $\Sa E_K$-open subsets of Zariski closed subsets of $V(K)$.
\end{enumerate}
\end{lemma}

\begin{proof}
It is clear that $(1)$ implies $(2)$.
Suppose $(2)$.
Following the proof of Lemma~\ref{lem:decomp} we obtain pairwise disjoint subvarieties $V_1,\ldots,V_k$ and $X_1,\ldots,X_k$ such that each $X_i$ is an \'etale open subset of $V_i(K)$ and $X = \bigcup_{i = 1}^{k} X_i$.
(Note that the $V_i$ may not be smooth as $K$ may not be perfect.)
By pairwise disjointness we have $X_i = V_i(K) \cap X$ for each $i$.
Thus each $X_i$ is definable.
\end{proof}

\begin{proposition}
\label{prop:closure-prop}
Suppose that $V,W,V_1,\ldots,V_k$ are $K$-varieties and $V \to W$ is a morphism.
\begin{enumerate}[leftmargin=*]
\item A finite union or finite intersection of \'ez subsets of $V(K)$ is an \'ez subset.
\item If $X \subseteq W(K)$ is an \'ez set then the preimage of $X$ under the map $V(K) \to W(K)$ induced by $V \to W$ is an \'ez set.
\item If $X$ is an \'ez subset of $K^{m + n}$ and $a \in K^m$ then $X_a = \{ b \in K^n : (a,b) \in X \}$ is an \'ez subset of $K^n$,
\item If $X_i$ is an \'ez subset of $V_i(K)$ for each $i \in \{1,\ldots,k\}$ then $X_1 \times \cdots \times X_k$ is an \'ez subset of $V_1(K) \times \cdots \times V_k(K) = (V_1 \times \cdots \times V_k)(K)$.
\end{enumerate}
\end{proposition}

\begin{proof}
$(1)$ 
Closure under finite unions is clear from the definitions.
For the second claim it suffices to suppose that $X_1,X_2$ are \'ez sets and show that $X_1 \cap X_2$ is an \'ez set.
Given $i \in \{1,2\}$ we suppose that $X^i_1,\ldots,X^i_k$ are basic \'ez sets such that $X_i = \bigcup_{j = 1}^k X^i_j$.
Then $X_1 \cap X_2 = \bigcup_{i,j \in \{1,\ldots,k\}} X^1_i \cap X^2_j$.
Thus we may suppose that $X_1$ and $X_2$ are basic \'ez sets.
It suffices to show that $X_1 \cap X_2$ is an \'etale open subset of a Zariski closed set.
Given $i \in \{1,2\}$ we let $Y_i$ be a Zariski closed subset of $V(K)$ and $O_i$ be an \'etale open subset of $V(K)$ such that $X_i = Y_i \cap O_i$.
Then $X_1 \cap X_2 = (Y_1 \cap Y_2) \cap (O_1 \cap O_2)$.
Note that $Y_1 \cap Y_2$ is Zariski closed and $O_1 \cap O_2$ is \'etale open.

\meno
$(2)$ 
Let $f$ be the induced map $V(K) \to W(K)$.
By $(1)$ we may suppose that $X$ is a basic \'ez subset of $W(K)$.
Suppose that $Y$ is a Zariski closed subset of $W(K)$ and $O$ is an \'etale open subset of $W(K)$ such that $X = Y \cap O$.
Then $f^{-1}(X) = f^{-1}(Y) \cap f^{-1}(O)$.
Note that $f^{-1}(Y)$ is Zariski closed and $f^{-1}(O)$ is \'etale open.

\meno
$(3)$
Let $g \colon  \Aa^n \to \Aa^{m +  n}$ be the morphism given by $x \mapsto (a,x)$.
Then $X_a$ is the preimage of $X$ under the map $K^n \to K^{m + n}$ induced by $g$.
Apply $(2)$.

\meno
$(4)$
For each $i \in \{ 1 , \ldots , n\}$ we let $\uppi_i$ be the projection $V_1(K) \times \cdots \times V_n(K) \to V_i(K)$.
Then
$$ X_1 \times \cdots \times X_n = \uppi^{-1}_1(X_1) \cap \cdots \cap \uppi^{-1}_n(X_n). $$
Apply $(1)$ and $(2)$.
\end{proof}

\begin{proposition}
\label{prop:qe-gen-0}
Every quantifier free definable subset of $K^n$ is \'ez.
\end{proposition}

\begin{proof}
Fix $f \in K[x_1,\ldots,x_n]$.
Then $\{ a \in K^n : f(a) = 0 \}$ is Zariski closed, hence \'ez.
Furthermore $\{a \in K^n : f(a) \ne 0 \}$ is Zariski open, hence \'ez.
Apply Proposition~\ref{prop:closure-prop}.
\end{proof}

We now prove Theorem A.

\begin{proof}
By Lemma~\ref{lem:ez-equiv} it suffices to show that $f(X)$ is a a finite union of \'etale open subsets of Zariski closed subsets of $W(K)$.
Suppose that $K$ is perfect, $f \colon  V \to W$ is a morphism of $K$-varieties, and $X$ is an \'ez subset of $V(K)$.
We show that $f(X)$ is an \'ez subset of $W(K)$.
We have $X = \bigcup_{i = 1}^k X_i$ for basic \'ez sets $X_1,\ldots,X_k$.
Then $f(X) = f(X_1)\cup\cdots\cup f(X_k)$.
By Proposition~\ref{prop:closure-prop}.1 we may suppose that $X$ is a basic \'ez subset of $V(K)$.
Let $V'$ be a closed subvariety of $V$ such that $X$ is an \'etale open subset of $V'(K)$.
After replacing $V$ with $V'$ and $f$ with the restriction $V' \to W$ we suppose that $X$ is an \'etale open subset of $V(K)$.

\meno
We apply induction on $\dim V$.
If $\dim V = 0$ then by Fact~\ref{fact:finite-var} $V$ is finite, hence $X$ is finite, so $X$ is Zariski closed.
Suppose $\dim V \ge 1$.
Let $V_1,\ldots,V_k$ be the irreducible components of $V$.
It suffices to show that each $f(V_i(K) \cap X)$ is an \'ez set.
By Proposition~\ref{prop:well-define} each $V_i(K) \cap X$ is an \'etale open subset of $V_i(K)$.
Therefore we may suppose that $V$ is irreducible.
Let $W'$ be the Zariski closure of $f(V)$ in $W$, so $V \to W'$ is dominating.
This implies that $W'$ is irreducible.
By Theorem F there is a dense open subvariety $U$ of $V$ such that $U(K) \to W'(K)$ is $\Sa E_K$-open.
Hence $f(U(K) \cap X)$ is an \'etale open subset of $W(K)$.
Let $V' := V \setminus U$.
Then $f(X) = f(U(K) \cap X) \cup f(V'(K) \cap X)$.
As $U$ is dense in $V$ we have $\dim V' < \dim V$, so by induction $f(V'(K) \cap X)$ is a \'ez set.
\end{proof}

Finally, we prove Corollary A.
Suppose that $K$ is perfect and $X$ is an existentially definable subset of $K^m$.
Let $x = (x_1,\ldots,x_m)$ and $y = (y_1,\ldots,y_n)$.
Then there is a quantifier-free formula $\phi(x,y)$ with parameters from $K$ such that for any $\alpha \in K^m$ we have $\alpha \in X$ if and only if $K \models \exists y \phi(\alpha,y)$.
Let $Y$ be the set of $(\alpha,\beta) \in K^{m + n}$ such that $K \models \phi(\alpha,\beta)$ and $\uppi \colon  K^{m + n} \to K^m$ be the coordinate projection.
Then $\uppi(Y) = X$.
Then $Y$ is \'ez by Proposition~\ref{prop:qe-gen-0} and $\uppi(Y)$ is \'ez by Theorem A.

\section{Sharpness and applications to large fields}
\label{section:sharp}

Fact~\ref{fact:fehm} is a theorem of Fehm~\cite{Fehm-subfield}. It was later generalized in~\cite{ansco-exis}.

\begin{fact}
\label{fact:fehm}
Suppose that $K$ is perfect and large.
Then $K$ does not existentially define an infinite proper subfield of $K$.
\end{fact}

We describe a topological proof of Fact~\ref{fact:fehm}.
Suppose $K$ is perfect and $L$ is an existentially definable infinite proper subfield of $K$.
By Corollary A $L$ has $\Sa E_K$-interior.
By Proposition~\ref{prop:subfield} below the \'etale open topology on $K$ is discrete.
By Fact~\ref{fact:large-potato} $K$ is not large.

\meno
Given $X \subseteq K$ we let $XX^{-1} = \{ \alpha/\beta : \alpha \in X, \beta \in X \setminus \{0\} \}$.
An \textbf{affine invariant} topology on $K$ is a topology that is invariant under any invertible affine transformation $K \to K$.
The \'etale open topology on $K$ is affine invariant by Fact~\ref{fact:first-paper-system}.1.

\begin{proposition}
\label{prop:subfield}
Suppose that $\uptau$ is an affine invariant topology on $K$, $U$ is a nonempty $\uptau$-open neighbouuprhood of zero, and $L$ is a proper subfield of $K$.
If $UU^{-1} \ne K$ then $\uptau$ is discrete.
If $L$ has $\uptau$-interior then $\uptau$ is discrete.
\end{proposition}

\begin{proof}
If $U = \{0\}$ then $\uptau$ is discrete, so we may suppose that $U$ contains a non-zero element.
Suppose $UU^{-1} \ne K$.
Fix $\alpha \in K \setminus UU^{-1}$.
Note that $\alpha \ne 0$.
Therefore $\alpha U \cap U$ is a $\uptau$-open neighbouuprhood of zero and $\alpha U \cap U = \{0\}$.
Hence $\uptau$ is discrete.
Now suppose that $L$ contains a nonempty $\uptau$-open $O \subseteq K$.
Fix $\alpha \in O$.
After replacing $O$ with $O - \alpha$ we suppose that $0 \in O$.
Then $OO^{-1} \subseteq L$ hence $OO^{-1} \ne K$.
Thus $\uptau$ is discrete.
\end{proof}

We also see that Corollary A is sharp.
Suppose $K$ is large and imperfect.
Let $F$ be the image of the Frobenius $K \to K$.
Then $F$ is existentially definable, infinite, and has empty $\Sa E_K$-interior by Proposition~\ref{prop:subfield}, so $F$ is not an \'ez subset of $K$.

\meno
Corollary~\ref{cor:subfield} follows by the arguments above.

\begin{corollary}
\label{cor:subfield}
If $K$ is \'ez then $K$ does not define an infinite proper subfield of $K$.
\end{corollary}

This allows us to easily give examples of large perfect fields which are not \'ez.
We follow \cite[Example 9]{Fehm-subfield}.
Let $L$ be a characteristic zero field and $L((x,y))$ be the fraction field of the formal power series ring $L[[x,y]]$.
Then $L((x,y))$ is large~\cite{Pop-henselian}, $L((x,y))$ defines $L[[x,y]]$~\cite[Theorem 3.34]{model-theoretic-algebra}, and by a theorem of Delon $L[[x,y]]$ defines the subfield $\Qq$~\cite[Theorem 2.1]{delon-formal-power-series}.
Therefore $L((x,y))$ is not \'ez.

\meno
We give two more applications to \'ez fields.
A theory $T$ is \textbf{one-cardinal} if for any $M \models T$ and infinite definable $X \subseteq M^n$ we have $|X| = |M|$~\cite[12.1]{Hodges}.
Algebraically, real, and $p$-adically closed fields are known to be one-cardinal.
We first generalize this fact.

\begin{corollary}
\label{cor:one-cardinal}
Suppose that $K$ is \'ez.
Then the theory of $K$ is one-cardinal.
\end{corollary}

\begin{proof}
Suppose that $X$ is an infinite definable subset of $K^m$.
As $X$ is infinite there is a coordinate projection $\uppi \colon  K^m \to K$ such that $\uppi(X)$ is infinite.
Then $\uppi(X)$ is \'ez and hence contains a nonempty \'etale open subset $O$ of $K$, fix $\gamma \in O$.
Let $Y$ be the set of $(\alpha,\beta) \in X^2$ such that $\uppi(\beta) \ne \gamma$ and $f \colon  Y \to K$ be given by $f(\alpha,\beta) = (\uppi(\alpha) - \gamma)/(\uppi(\beta) - \gamma)$.
By Proposition~\ref{prop:subfield} $f$ is surjective.
We have shown that for every infinite definable $X \subseteq K^m$ there is $\gamma \in K$ and a coordinate projection $\uppi \colon  K^m \to K$ such that the map
\[
f\colon \{(\alpha,\beta)\in K^2 : \uppi(\beta) \ne \gamma \} \longrightarrow K, \quad\quad f(\alpha,\beta) = \frac{\uppi(\alpha) - \gamma}{\uppi(\beta) - \gamma}
\]
is surjective.
By elementary transfer the same holds for any model of the theory of $K$.
Hence the theory of $K$ is one-cardinal.
\end{proof}

A field topology on $K$ is a \textbf{V-topology} if and only if it is induced by a non-trivial absolute value or valuation.
We refer to \cite[Appendix B]{EP-value} for background on V-topologies.

\begin{corollary}
\label{cor:refine-V}
Suppose $K$ is large and perfect and $\uptau$ is a  V-topology on $K$.
Then the following are equivalent:
\begin{enumerate}
\item the \'etale open topology on $K$ refines $\uptau$, 
\item the \'etale open topology on $V(K)$ refines the $\uptau$-topology for any $K$-variety $V$,
\item there is an infinite existentially definable subset of $K$ which is not $\uptau$-dense in $K$.
\end{enumerate}
Suppose furthermore that $K$ is \'ez.
Then $(1) - (3)$ above hold if and only if there is an infinite definable subset of $K$ which is not $\uptau$-dense.
\end{corollary}

\begin{proof}
We show that $(1) - (3)$ are equivalent, the last claim follows from our proof.
The equivalence of $(1)$ and $(2)$ holds without any assumptions on $K$, see \cite[Lemma 4.8]{firstpaper}.
The following is also shown in \cite[Lemma 6.9]{firstpaper}: the \'etale open topology on $K$ refines $\uptau$ if and only if some nonempty \'etale open subset $U$ of $K$ is not $\uptau$-dense.
Let $U$ be such a set.
Fix $p \in U$.
Then there is an \'etale morphism of $K$-varieties $f \colon  V \to \Aa^1$ such that $p \in f(V(K)) \subseteq U$.
Then $f(V(K))$ is existentially definable, infinite, and not $\uptau$-dense.
Suppose that $X$ is an infinite existentially definable subset of $K$ which is not $\uptau$-dense.
By Theorem A we have $X = U \cup Y$ where $U$ is \'etale open and $Y$ is finite.
Then $U$ is nonempty, note that $U$ is not $\uptau$-dense.
\end{proof}

\section{Henselian and t-Henselian fields}
\label{section:Henselian}
We show that characteristic zero t-Henselian fields are \'ez, proving Theorem C.3.
We suppose that $K$ is characteristic zero t-Henselian.
If $K$ is algebraically closed then quantifier elimination and Proposition~\ref{prop:qe-gen-0} show that every definable subset of $K^m$ is \'ez.
\textbf{Hence we may suppose that $K$ is not algebraically closed.}

\meno
We first suppose that $K$ is a Henselian field.
First recall that Henselian fields are large~\cite[1.A.3]{Pop-little}.
By Fact~\ref{fact:hd}.2 the $\Sa E_K$-topology on each $K^m$ agrees with the valuation topology.
By Lemma~\ref{lem:ez-equiv} it is enough to show that every definable set is a finite union of valuation open subsets of Zariski closed sets. This has been obtained by van den Dries~\cite{lou-dimension}, his proof makes crucial use of quantifier eliminations due to Delon (see~ \cite[Pg 191, 3.3, and 3.7]{lou-dimension}).

\meno
We use Fact~\ref{fact:two-poly} in the general t-Henselian case.

\begin{fact}
\label{fact:two-poly}
If $V$ is a subvariety of $\Aa^m$ then $V(K)=Z(f)\setminus Z(g)$ for some $f,g\in K[x_1,\ldots,x_m]$.
\end{fact}

Fact~\ref{fact:two-poly} follows from Fact~\ref{fact:one-poly} and the definition of a subvariety.

\meno
We now suppose that $K$ is t-Henselian of characteristic zero.
By \cite[Theorem 7.2.b]{Prestel1978} $K$ has a Henselian elementary extension $K^*$.
Let $\uptau$, $\uptau^*$ be the t-Henselian topology on $K$, $K^*$, respectively.
Given a $K$-definable set $Y$ we let $Y^*$ be the $K^*$-definable set defined by the same formula as $X$.
By \cite[Lemma~7.5]{Prestel1978} there is a definable bounded open neighbourhood $U$ of zero.
Let $\Sa B:=(\alpha U+\beta:\alpha\in K^\times,\beta\in K)$.
Note that $\Sa B$ is a definable family of sets.
By definition of a t-Henselian topology is a basis for $\uptau$.
Furthermore $\Sa B^*= (\alpha U^* + \beta : \alpha,\beta\in K^*, \alpha\ne 0)$ is a $K^*$-definable basis for $\uptau^*$.
The t-Henselian topology on $K^m$ is the product topology and hence admits a definable basis, likewise for $K^*$.

\meno
By elementary transfer $K$ is large so it is enough to show that every definable set is \'ez.
Suppose that $X\subseteq K^m$ is definable and let $X^*$ be the $K^*$-definable set defined by the same formula.
By the previous paragraph $X^*$ is \'ez.
By Fact~\ref{fact:two-poly} and Lemma~\ref{lem:decomp} there are $f_1,g_1,\ldots,f_k,g_k\in K^*[x_1,\ldots,x_m]$ such that
\begin{enumerate}
\item $X^*\cap [Z(f_i)\setminus Z(g_i)]$ is $\uptau^*$-open for all $i \in \{1,\ldots,k\}$,
\item $X^* = (X^*\cap[Z(f_1)\setminus Z(g_1)])\cup\cdots\cup(X^*\cap[Z(f_k)\setminus Z(g_k)])$.
\end{enumerate}
Let $d$ be the maximum of the degrees of $f_1,g_1,\ldots,f_k,g_k$.
By elementary transfer and definability of $\uptau$ and $\uptau^*$ there are $f'_1,g'_1,\ldots,f'_k,g'_k\in K[x_1,\ldots,x_m]$ of degree $\le d$ such that $(1),(2)$ above hold with $X$ in place of $X^*$.
Hence $X$ is \'ez.

\section{Frobenius fields}
\label{section:frobenius}
In this section we prove Theorem~\ref{thm:pac}.
This completes the proof of Theorem C.

\begin{theorem}
\label{thm:pac}
Perfect Frobenius fields are \'ez.
\end{theorem}

Frobenius fields are by definition $\pac$, and $\pac$ fields are large~\cite[1.A.1]{Pop-little}.
Hence it is enough to show that every definable set is an \'ez set.

\begin{proposition}
\label{prop:qe-gen}
Suppose that $K$ is large and perfect.
Suppose that $\Sa L$ is an expansion of the language of rings by relation symbols, $\Sa K$ is an $\Sa L$-structure which expands $K$ by definitions, and $\Sa K$ is model complete.
Suppose $\{ \alpha \in K^m : \Sa K \models R(\alpha) \}$ and $\{ \alpha \in K^m : \Sa K \models \neg R(\alpha) \}$ are \'ez sets for any $n$-ary relation symbol $R \in \Sa L$.
Then $K$ is \'ez.
\end{proposition}

\begin{proof}
Suppose that $X$ is an $\Sa L$-definable subset of $K^m$.
Then there is a quantifier free $\Sa L$-definable subset $Y$ of $K^{m + n}$ such that $\uppi(Y) = X$, where $\uppi$ is the coordinate projection $K^{m + n} \to K^m$.
By Theorem A it suffices to show that a quantifier free $\Sa L$-definable subset of $K^m$ is \'ez.
By Proposition~\ref{prop:closure-prop}.1 it suffices to show that any atomic or negated atomic $\Sa L$-formula $\phi(x_1,\ldots,x_m)$ defines an \'ez subset of $K^m$.
Let $x = (x_1,\ldots,x_m)$.
By Proposition~\ref{prop:qe-gen-0} it suffices to consider two kinds of formulas:
\begin{enumerate}
\item \hspace{.15cm} $R(f_1(x),\ldots,f_{n}(x))$ for an $n$-ary $R \in \Sa L$ and $f_1,\ldots,f_{n} \in K[x]$,
\item $\neg R(f_1(x),\ldots,f_{n}(x))$ for an $n$-ary $R \in \Sa L$ and $f_1,\ldots,f_{n} \in K[x]$.
\end{enumerate}
We treat case $(1)$, the second case follows by the same argument.
Let $f = (f_1,\ldots,f_n)$.
Then
$$ \{ \alpha \in K^m : \Sa K \models R(f_1(\alpha),\ldots,f_{n}(\alpha)) \} = f^{-1} \left( \{ \beta \in K^n : \Sa K \models R(\beta) \} \right). $$
Apply Proposition~\ref{prop:closure-prop}.2.
\end{proof}

Fact~\ref{fact:irreducible} is  \cite[Corollary 3.7]{with-anand}.

\begin{fact}
\label{fact:irreducible}
The set of $(\alpha_0,\ldots,\alpha_{m - 1}) \in K^m$ such that $t^m + \alpha_{m - 1}t^{m - 1} + \cdots + \alpha_1 t + \alpha_0$ is separable and irreducible in $K[t]$ is \'etale open.
\end{fact}

We also apply Fact~\ref{fact:frobenius}.
Fact~\ref{fact:frobenius} was proven in unpublished but very influential work of Cherlin, van den Dries, and Macintyre \cite[Theorem 41]{pac-unpub}.
(Frobenius fields are referred to as ``Iwasawa fields" in \cite{pac-unpub}.)

\begin{fact}
\label{fact:frobenius}
Let $\Sa L$ be the expansion of the language of rings by an $m$-ary relation symbol $R_m$ for all $m \ge 2$ and $\Sa K$ be the expansion of $K$ to an $\Sa L$-structure where for all $\alpha_0,\ldots,\alpha_{m - 1} \in K$:
$$ \Sa K \models R_m(\alpha_0,\ldots,\alpha_{m -1}) \quad \Longleftrightarrow \quad K \models \exists t (t^m + \alpha_{m - 1}t^{m -1} + \cdots + \alpha_1t + \alpha_0 = 0). $$
If $K$ is a perfect Frobenius field then $\Sa K$ admits quantifier elimination. 
\end{fact}

Theorem~\ref{thm:pac} follows from Fact~\ref{fact:frobenius}, Proposition~\ref{prop:qe-gen}, and Proposition~\ref{prop:Sn} below.

\begin{proposition}
\label{prop:Sn}
Suppose $K$ is perfect.
For any $m \ge 2$ both
\begin{align*}
 X_m &:= \{ (\alpha_0,\ldots,\alpha_{m - 1}) \in K^m : K\models \forall t (t^m + \alpha_{m - 1}t^{m - 1} + \cdots + \alpha_1 t + \alpha_0 \ne 0) \}, \text{  and} \\
 Y_m &:= \{ (\alpha_0,\ldots,\alpha_{m - 1}) \in K^m : K\models \exists t (t^m + \alpha_{m - 1}t^{m - 1} + \cdots + \alpha_1 t + \alpha_0 = 0) \}
\end{align*}
are \'ez.
\end{proposition}

For each $\alpha = (\alpha_0,\ldots,\alpha_{m - 1}) \in K^m$ we let $p_\alpha \in K[t]$ be $t^m + \alpha_{m - 1}t^{m-1} + \cdots + \alpha_1t + \alpha_0$.
\begin{proof}

Each $Y_m$ is \'ez by Corollary A.
We apply induction on $m \ge 2$ to show that $X_m$ is \'ez.
As $K$ is perfect an irreducible $p_\alpha$ is also separable.
A quadratic or cubic polynomial does not have a root if and only if it is irreducible, so by Fact~\ref{fact:irreducible} $X_2$ and $X_3$ are both $\Sa E_K$-open, hence \'ez.
Suppose $m \ge 4$.
If $a \in K^m$ and $p_\alpha$ does not have a root in $K$ then either:
\begin{enumerate}
\item $p_\alpha$ is irreducible, or
\item there is $k \in \{2,\ldots,m - 2\}$, $\beta \in K^k$, and 
$\gamma \in K^{m - k}$ such that $p_\alpha = p_\beta p_\gamma$ and neither $p_\beta$ nor $p_\gamma$ has a root in $K$.
\end{enumerate}
By Fact~\ref{fact:irreducible} the set of $\alpha \in K^m$ such that $p_\alpha$ is irreducible is \'etale open, so it suffices to show that the set of $\alpha \in K^m$ satisfying $(2)$ is an \'ez set.
It is enough to fix $k \in \{2,\ldots,m - 2\}$ and show that
$$ \{ \alpha \in K^m : \exists\hspace{.1cm}(\beta,\gamma) \in K^k \times K^{m - k} [ (p_\alpha = p_\beta p_\gamma) \land (\beta \in X_k) \land (\gamma \in X_{m - k}) ] \} $$
is an \'ez set.
By Theorem A it suffices to show that
$$ \{ (\alpha,\beta,\gamma) \in K^m \times K^k \times K^{m - k} : (p_\alpha = p_\beta p_\gamma) \land (\beta \in X_k ) \land (\gamma \in X_{m - k}) \} $$
is an \'ez set.
By Proposition~\ref{prop:closure-prop}.1 it suffices to show that both 
\begin{enumerate}
\item $\{ (\alpha,\beta,\gamma) \in K^m \times K^k \times K^{m - k} : p_\alpha
= p_\beta p_\gamma \}$
\item and $K^m \times X_k \times X_{m - k}$
\end{enumerate}
are \'ez subsets of $K^m \times K^k \times K^{m - k}$.
The first set is Zariski closed, hence \'ez.
By induction $X_m$ and $X_{m - k}$ are both \'ez.
Apply Proposition~\ref{prop:closure-prop}.4.
\end{proof}

\section{Dimension of \'ez sets, proof of Theorem B}
\label{section:ez}
We prove some natural facts about dimension of \'ez sets under the assumption that $K$ is large.
Given a $K$-variety $V$ and a subset $X$ of $V$ we let $\overline{X}$ be the Zariski closure of $X$ in $V$.
Recall that $\dim X = \dim \overline{X}$ by definition.  
We first prove Lemma~\ref{lem:ez basic-0} which shows that the results of this section apply to existentially definable sets in  perfect large fields and arbitrary definable sets in  \'ez fields.

\begin{lemma}
\label{lem:ez basic-0}
Suppose that $V$ is a $K$-variety, $X$ is a definable subset of $V(K)$, and either:
\begin{enumerate}
\item $X$ is existentially definable, or
\item $K$ is \'ez.
\end{enumerate}
Then $X$ is an \'ez subset of $V(K)$.
\end{lemma}

\begin{proof}
Suppose $(2)$.
Let $V_1,\ldots,V_k$ be affine open subvarieties of $V$ that cover $V$.
By Proposition~\ref{prop:closure-prop}.1 it suffices to show that each $X \cap V_i(K)$ is an \'ez set.
We suppose that $V$ is affine.
Let $V \to \Aa^m$ be a closed immersion.
Let $X'$ be the image of $X$ under $V(K) \to K^m$, then $X'$ is a definable set and is hence an \'ez set.
Note that $X$ is the preimage of $X'$ under $V(K) \to K^m$ and apply Proposition~\ref{prop:closure-prop}.2.
If $X$ is existentially definable then the relevant objects are existentially definable, and the same argument shows that $X$ is \'ez.
\end{proof}

\begin{theorem}
\label{thm:dim}
Suppose that $K$ is perfect and large, $V$ is a $K$-variety, and $X$ is a nonempty \'ez subset of $V(K)$.
Let $W_1,\ldots,W_k$ be smooth irreducible subvarieties of $V$, and $X_1,\ldots,X_k$ be such that each $X_i$ is a nonempty  \'etale open subset of $W_i(K)$ and $X = X_1\cup\cdots\cup X_k$.
Then $\dim X$ is the maximum of $\dim W_1,\ldots,\dim W_k$.
\end{theorem}

Lemma~\ref{lem:decomp} ensures that such $W_i$ and $X_i$ exist.

\begin{proof}
By Fact~\ref{fact:with-anand} each $U_i$ is Zariski dense in $W_i$ and is hence Zariski dense in $\overline{W}_i$.
Thus $\overline{X} = \overline{W_1}\cup\cdots\cup\overline{W_k}$.
By Fact~\ref{lem:dim-fin-union} we have
$$ \dim X = \max \{ \dim \overline{W_1},\ldots,\dim\overline{W_k} \} = \max \{ \dim W_1,\ldots \dim W_k \}. $$
\end{proof}

\begin{lemma}
\label{lem:int-vs-nwd}
Suppose that $K$ is large, $V$ is a smooth irreducible $K$-variety, and $X$ is a nonempty \'ez subset of $V(K)$.
Then $X = O \cup Y$ where $O$ is a definable \'etale open subset of $V(K)$ and $Y$ is not Zariski dense in $V(K)$.
\end{lemma}

Lemma~\ref{lem:int-vs-nwd} applies in particular to $V = \Aa^m$.
Note that an \'ez subset of $V(K)$ agrees Zariski-locally at the generic point of $V$ with a (possibly empty) \'etale open subset of $V(K)$.

\begin{proof}
Note that $V(K)$ is Zariski dense in $V$ by Fact~\ref{fact:with-anand}.
Let $V_1,\ldots,V_k$ be closed subvarieties of $V(K)$ and $X_1,\ldots,X_k$ be such that each $X_i$ is an \'etale open subset of $V_i(K)$ and $X = \bigcup_{i = 1}^k X_i$.
By irreducibility of $V$ each $V_i$ is either nowhere Zariski dense in $V$ or agrees with $V$.
Let $I$ be the set of  $i \in \{1,\ldots,n\}$ such that $V_i = V$.
Then $X_i$ is an \'etale open subset of $V(K)$ when $i \in I$ and $X_i$ is not Zariski dense in $V(K)$ when $i \notin I$.
Let $U = \bigcup_{i \in I} X_i$ and $Y = \bigcup_{i \notin I} X_i$.
\end{proof}

\begin{proposition}
\label{prop:smooth}
Suppose that $K$ is large, $V$ is a smooth irreducible $K$-variety, and $X$ is a nonempty \'ez subset of $V(K)$.
Then $X$ has $\Sa E_K$-interior in $V(K)$ if and only if $\dim X = \dim V$.
\end{proposition}

Again, Proposition~\ref{prop:smooth} applies to $V = \Aa^m$.

\begin{proof}
By irreducibility $\dim X = \dim V$ if and only if $X$ is Zariski dense in $V$.
By Lemmas~\ref{lem:ez basic-0} and \ref{lem:int-vs-nwd} we have $X = U \cup Y$ where $U \subseteq V(K)$ is \'etale open and $Y \subseteq V(K)$ is not Zariski dense.
By Fact~\ref{fact:with-anand} $Y$ has empty $\Sa E_K$-interior in $V(K)$.
Hence $X$ has $\Sa E_K$-interior in $V(K)$ if and only if $U \ne \emptyset$.
Again by Fact~\ref{fact:with-anand} $U \ne \emptyset$ if and only if $X$ is Zariski dense in $V$.
\end{proof}

Corollary~\ref{cor:lou} follows directly from Proposition~\ref{prop:smooth},

\begin{corollary}
\label{cor:lou}
Suppose that $K$ perfect and large and $V$ is a $K$-variety.
Then every \'ez subset of $V(K)$ has nonempty $\Sa E_K$-interior in its Zariski closure.
\end{corollary}

\begin{lemma}
\label{lem:ez-mod}
Suppose that $K$ is large and perfect, $V$ is a nonempty irreducible $K$-variety, and $X$ is an \'ez subset of $V(K)$.
Then there is a smooth subvariety $W$ of $V$, a nonempty \'etale open subset $O$ of $W(K)$, and a dense open subvariety $U$ of $V$ such that $O = X \cap U(K)$ and $\dim X \setminus O < \dim X$.
\end{lemma}

In particular an \'ez subset of $V(K)$ is, modulo a set of lower dimension, an \'etale open subset of the $K$-points of a smooth subvariety of $V$.
The elements of $O$ can be reasonably considered to be smooth points of $X$, so Lemma~\ref{lem:ez-mod} informally shows almost every point of $X$ is smooth.

\begin{proof}
By Lemma~\ref{lem:decomp} there are pairwise disjoint smooth irreducible subvarieties $V_1,\ldots,V_k$ of $V$ and $X_1,\ldots,X_k$ such that each $X_i$ is a nonempty \'etale open subset of $V_i(K)$ and $X = X_1\cup\cdots\cup X_k$.
Let $\dim X = d$.
By Theorem~\ref{thm:dim} we have $d = \max \{ \dim V_1,\ldots,\dim V_k \}$ and $\dim V_i = \dim X_i$ for each $i$.
We may suppose that there is $\ell \in \{1,\ldots,k\}$ such that $\dim V_i = d$ when $i \le \ell$ and $\dim V_i < d$ when $\ell < i$.
Let $Z = \bigcup_{i = 1}^{k} \overline{V_i} \setminus V_i$.
By Fact~\ref{fact:sub-frontier} each $\overline{V_i} \setminus V_i$ is a closed subvariety of $V$, so $Z$ is a closed subvariety of $V$.
By pairwise disjointness we have $V_i \setminus Z = V_i \setminus \bigcup_{j \ne i} \overline{V_j}$ for all $i \in \{1,\ldots,k\}$.
It follows that each $X_i \setminus Z$ is a (possibly empty) \'etale open subset of $X$.
By Fact~\ref{fact:constructible} and Fact~\ref{lem:dim-fin-union} we have
\begin{align*}
\dim Z &= \max \{ \dim \overline{V_2} \setminus V_2 , \ldots , \dim \overline{V_k} \setminus V_k \} \\ &< \max \{ \dim V_2 , \ldots , \dim V_k \} = d.
\end{align*}
Hence $X_i \setminus Z$ is nonempty when $i \le \ell$.
Let $U = V \setminus (Z \cup V_{\ell + 1} \cup \cdots \cup V_k)$, so $U$ is a dense open subvariety of $V$.
If $i \le \ell$ then $V_i \cap U$ is disjoint from $\overline{V_j}$ for $j \ne i$.
Let $O = X \cap U =  (X_1 \cap U) \cup \cdots \cup (X_\ell \cap U)$.
By Fact~\ref{fact:with-anand} $X_i \cap U$ is nonempty when $i \le \ell$.
Let $W = \bigcup_{i = 1}^{\ell} V_i \cap U$, note that $W$ is isomorphic as a $K$-variety to the disjoint union of the $V_i \cap U$.
If $i \le \ell$ then $U$ is  Zariski dense in $V_i$.
Hence $W$ is smooth as each $V_i$ is smooth.
Each $X_i \cap U$ is an $\Sa E_K$-open subset of $W(K)$, hence $O$ is an $\Sa E_K$-open subset of $W(K)$.
We have
$$ \dim V \setminus U = \max \{ \dim Z, \dim V_{\ell + 1},\ldots , \dim V_k \} < d. $$
Hence $\dim X \setminus O < d$.
\end{proof}

\begin{proposition}
\label{prop:dim-open}
Suppose that $K$ is perfect and large, $V$ is a $K$-variety, $X \subseteq Y$ are \'ez subsets of $V(K)$, and $\dim X = \dim Y$.
Then $X$ has nonempty $\Sa E_K$-interior in $Y$.
\end{proposition}

The converse to Proposition~\ref{prop:dim-open} fails, e.g. let $X = \{(0,1)\}$ and $Y = X \cup \{ (t,0) : t \in K \}$.

\begin{proof}
Applying Lemma~\ref{lem:ez-mod} to $Y$ we let $W$ be a smooth subvariety of $V$, $O$ be an \'etale open subset of $W(K)$, and $U$ be a dense open subvariety of $V$ such that $Y \cap U = O$ and $\dim Y \setminus O < \dim X$.
By Fact~\ref{lem:dim-fin-union} $ \dim X = \max \{ \dim X \cap O, \dim X \setminus O \} $, so $\dim X \cap O = \dim X$.
By Proposition~\ref{prop:smooth} $X \cap O$ has nonempty $\Sa E_K$-interior in $W(K)$, so $X \cap O$ has nonempty $\Sa E_K$-interior in $Y$.
\end{proof}

We give an application to definable groups in \'ez fields.
Recall that a \textbf{$K$-algebraic group} is a group object in the category of $K$-varieties.

\begin{corollary}
\label{cor:def-group}
Suppose that $K$ is large, $G$ is a $K$-algebraic group, and $H$ is an \'ez subgroup of $G(K)$.
Then $H$ is an \'etale open subgroup of its Zariski closure in $G(K)$.
Thus if $K$ is \'ez then any definable subgroup of $G(K)$ is an \'etale open subgroup of its Zariski closure.
\end{corollary}

Van den Dries showed that if $K$ is a characteristic zero Henselian field then any definable subgroup of $\mathrm{GL}_n(K)$ is a valuation open subgroup if its Zariski closure~\cite[2.20]{lou-dimension}.

\begin{proof}
Let $W$ be the Zariski closure of $H$ in $G$.
Then $W$ is a $K$-algebraic subgroup of $G$, so $W(K)$ is a Zariski closed subgroup of $G(K)$.
Note that if $\alpha \in W(K)$ then $x \mapsto \alpha x$ gives a $K$-variety isomorphism $W \to W$, and hence induces an $\Sa E_K$-homeomorphism $W(K) \to W(K)$.
By Corollary~\ref{cor:lou} $H$ contains a nonempty \'etale open $O \subseteq W(K)$.
We have $H = \bigcup_{a \in H} \alpha O$, so $H$ is an \'etale open subset of $W(K)$.
\end{proof}

We now discuss large simple fields.
We assume some familiarity with forking in simple theories. Readers can refer to~\cite{wagner-simple} for more details on this subject.
We refer to \cite[Section 3]{Pillay-definability-simple} for background on $f$-generics in groups definable in simple theories, note that Pillay uses ``generic" where we use ``$f$-generic".

\begin{corollary}
\label{cor:bounded-PAC}
Suppose that $K$ is perfect, bounded, and $\mathrm{PAC}$ and $X$ is a definable subset of $K^n$.
Then $X$ is $f$-generic for $(K^n,+)$ if and only if $X$ has nonempty $\Sa E_K$-interior in $K^n$.
\end{corollary}

Bounded $\pac$ fields are simple~\cite{Chatzidakis} and infinite simple fields are conjectured to be bounded $\pac$, see for example~\cite{supersimple-field}.
Pseudofinite fields and infinite algebraic extensions of finite fields are perfect, bounded, and $\pac$.
Corollary~\ref{cor:bounded-PAC} follows from Corollary~\ref{cor:exis-PAC} below and the fact that perfect bounded $\mathrm{PAC}$ fields are \'ez.

\begin{corollary}
\label{cor:exis-PAC}
Suppose that $K$ is perfect, large, and simple.
Let $X\subseteq K^n$ be \'ez.
Then $X$ is $f$-generic for $(K^n,+)$ if and only if $X$ has nonempty $\Sa E_K$-interior.
Hence an existentially definable subset of $K^n$ is $f$-generic for $(K^n,+)$ if and only if it has nonempty $\Sa E_K$-interior.
\end{corollary}

Corollary~\ref{cor:exis-PAC} follows from Lemma~\ref{lem:int-vs-nwd}, Fact~\ref{fact:with-anand-1}, and Lemma~\ref{lem:dense-generic} below.
Fact~\ref{fact:with-anand-1} is proven in \cite{with-anand}.

\begin{fact}
\label{fact:with-anand-1}
If $K$ is large and simple then any nonempty definable \'etale open subset of $K^n$ is $f$-generic for $(K^n,+)$.
\end{fact}

\begin{lemma}
\label{lem:dense-generic}
Suppose that $K$ is infinite and simple, $X$ is a definable subset of $K^n$, and $X$ is not Zariski dense in $K^n$.
Then $X$ is not $f$-generic for $(K^n,+)$.
\end{lemma}

In the proof below we use ``$f$-generic" for ``$f$-generic for $(K^n,+)$".

\begin{proof}
It suffices to show that the Zariski closure of $X$ is not $f$-generic.
Thus we may suppose that $X$ is Zariski closed, in particular $X$ is quantifier free definable.
Let $\mathbf{K}$ be a highly saturated elementary expansion of $K$ and $\mathbf{K}^{\mathrm{alg}}$ be the algebraic closure of $K$.
Let $Y$ be the subset of $\mathbf{K}^n$ defined by the same formula as $X$ and $Y'$ be the $\mathbf{K}^{\mathrm{alg}}$-definable set defined by the same (quantifier free) formula as $X$.
Fix $a \in \mathbf{K}^n$ such that the type of $a$ over $K$ is $f$-generic.
It is enough to show that $a + Y$ divides over $K$.
Let $(a_i)_{i \in I}$ be a $\mathbf{K}$-Morley sequence in $a$ over $K$.
Then $(a_i)_{i \in I}$ is also a Morley sequence in $\mathbf{K}^{\mathrm{alg}}$.
Then $a + Y'$ divides in $\mathbf{K}^{\mathrm{alg}}$ over $K$ as $\dim a + Y' < n$, so by Kim's lemma, $(a_i)_{i \in I}$ witnesses dividing in $\mathbf{K}^{\mathrm{alg}}$.
It is now easy to see that $a + Y$ divides over $K$.
\end{proof}

\section{Algebraic boundedness, proof of Theorem D}
In this section we show that \'ez fields are algebraically bounded.
Let $Z$ be a $K$-variety.
Given a subvariety $W$ of $Z \times \Aa^1$ we let $W_\alpha$ be the scheme-theoretic fiber of $W$ over $\alpha \in Z$.
Given a subset $X$ of $Z(K) \times K$ we let $X_\alpha$ be the set-theoretic fiber of $X$ above $\alpha \in Z(K)$, i.e. $\{ \beta \in K : (\alpha,\beta) \in X \}$.
Recall that the $K$-points of the scheme-theoretic fiber agree with the set-theoretic fiber of the $K$-points, i.e. $W_\alpha(K) = W(K)_\alpha$.

\begin{theorem}
\label{thm:ez bounded}
Suppose that $K$ is \'ez, $Z$ is a $K$-variety, and $X \subseteq Z(K) \times K$ is definable.
Then there are closed subvarieties $V_1,\ldots,V_\ell$ of $Z \times \Aa^1$ such that for any $\alpha \in K^m$ with $0 < |X_\alpha| < \infty$ there is $i$ such that $(V_i)_\alpha$ is finite and contains $X_\alpha$.
\end{theorem}

We first explain how Theorem~\ref{thm:ez bounded} implies that \'ez fields are algebraically bounded.
Algebraically closed fields are algebraically bounded~\cite[2.9]{lou-dimension}, so we suppose that $K$ is \'ez and not algebraically closed.
Let $X \subseteq K^m \times K$ be definable, and $V_1,\ldots,V_\ell$ be closed subvarieties  of $\Aa^m \times \Aa^1$ as above.
Applying Fact~\ref{fact:one-poly} we obtain for each $V_i$ a polynomial $f_i$ such that $V_i(K) = \{ \alpha \in K^m \times K : f_i(\alpha) = 0 \}$.
Algebraic boundedness follows.

\meno
We first prove Lemma~\ref{lem:fiber-open}.

\begin{lemma}
\label{lem:fiber-open}
Suppose that $K$ is large, $Z$ is a $K$-variety, $W$ is a subvariety of $Z \times \Aa^1$, $O$ is a nonempty \'etale open subset of $W(K)$, and $\alpha \in Z(K)$ lies in the image of the projection $O \to Z(K)$.
Then $O_\alpha$ is finite if and only if $W_\alpha$ is finite.
\end{lemma}

\begin{proof}
The right to left implication is trivial.
Suppose that $W_\alpha$ is infinite.
Then $W_\alpha$ is a dense open subvariety of $\Aa^1$, so $W_\alpha(K)$ is a cofinite subset of $K$.
Let $W_\alpha \to W$ be the morphism given by $x \mapsto (x,\alpha)$.
Then $O_\alpha$ is the preimage of $O$ under the induced map $W_\alpha(K) \to W(K)$.
Therefore $O_\alpha$ is a nonempty \'etale open subset of $W_\alpha(K)$, hence $O_\alpha$ is an \'etale open subset of $K$.
Hence $O_\alpha$ is infinite by largeness.
\end{proof}

\begin{lemma}
\label{lem:ez bounded}
Suppose that $K$ is large, $Z$ is a $K$-variety, and $X \subseteq Z(K) \times K$ is an \'ez set.
Then $\{ \alpha \in Z(K) : 0 < |X_\alpha| < \infty \}$ is definable and there is $n$ such that if $\alpha \in Z(K)$ and $X_\alpha$ is finite then $|X_\alpha| \le n$.
Particularly, if $K$ is \'ez then $K$ eliminates $\exists^\infty$.
\end{lemma}

\begin{proof}
The second claim follows easily from the first claim, so we only prove the first claim.
Let $W_1,\ldots,W_k$ be closed subvarieties of $Z \times \Aa^1$ and $X_1,\ldots,X_k$ be such that each $X_i$ is a nonempty definable \'etale open subset of $W_i(K)$ and $X = X_1\cup\cdots\cup X_k$.
For each $i$ let $Y_i$ be the set of $\alpha \in Z$ such that $|(W_i)_\alpha| < \infty$.
By Fact~\ref{fact:constructible} each $Y_i$ is a Zariski open subset of $Z$, hence $Y_i \cap Z(K)$ is definable.
For each $i$ let $P_i$ be the set of $\alpha \in Z(K)$ such that $\alpha \in \uppi(X_i)$ implies $\alpha \in Y_i$.
Note that each $P_i$ is definable.
Lemma~\ref{lem:fiber-open} shows that for any $\alpha \in Z(K)$, $(X_i)_\alpha$ is finite if and only if $\alpha \in P_i$.
Therefore $0 < |X_\alpha| < \infty$ if and only if $\alpha \in \uppi(X)$ and $\alpha \in P_i$ for all $i$.
Finally note that $\uppi(X)$ is definable. 

\meno
Fact~\ref{fact:constructible} shows that for each $i$ there is $n_i$ such that if $\alpha \in Z$ and $|(W_i)_\alpha| < \infty$ then $|(W_i)_\alpha| \le n_i$.
By what is above we have  $|X_\alpha| < \infty$ if and only if there is $I \subseteq \{1,\ldots,k\}$ such that $X_a \subseteq \bigcup_{i \in I} (W_i)_\alpha$ and $|(W_i)_\alpha| < \infty$ for all $i \in I$.
Thus $|X_\alpha| < \infty$ implies $|X_\alpha| < n_1 + \cdots + n_k$ for all $\alpha \in Z(K)$.

\end{proof}

We now prove Theorem~\ref{thm:ez bounded}.

\begin{proof}
By Lemma~\ref{thm:ez bounded} $\{ \alpha \in Z(K) : 0 < |X_\alpha| < \infty \}$ is definable.
After possibly replacing $X$ with $\{ (\alpha,\beta) \in X : 0 < |X_\alpha| < \infty \}$ we suppose that $X_\alpha$ is finite for all $\alpha \in Z(K)$.
Applying Lemma~\ref{lem:decomp} we fix smooth irreducible subvarieties $W_1,\ldots,W_k$ of $Z \times \Aa^1$ and $X_1,\ldots,X_k$ such that each $X_i$ is an \'etale open subset of $W_i(K)$ and $X =X_1\cup\cdots\cup X_n$.
By Fact~\ref{fact:with-anand} each $X_i$ is Zariski dense in $W_i$.
Let $\uppi \colon  Z \times \Aa^1 \to Z$ be the projection.

\begin{claim}
Fix $i$ and let $Y_i = \{ \alpha \in \uppi(W_i): |(W_i)_\alpha| < \infty\}$.
Then $\dim \uppi(W_i) \setminus Y_i < \dim \uppi(W_i)$.
\end{claim}

\begin{claimproof}
By Fact~\ref{fact:constructible} $\uppi(W_i)$ is constructible and $Y_i$ is a Zariski open subset of $\uppi(W_i)$.
Note that $\uppi(X_i)$ is Zariski dense in $\uppi(W_i)$ as $X_i$ is Zariski dense in $W_i$.
By Lemma~\ref{lem:fiber-open} $\uppi(X_i) \subseteq Y_i$, so $Y_i$ is Zariski dense in $\uppi(X_i)$.
By Fact~\ref{fact:constructible} $\dim \uppi(W_i) \setminus Y_i < \dim \uppi(W_i)$.
\end{claimproof}

\meno
Let $W = \bigcup_{i = 1}^{k} W_i$.
Then $X$ is Zariski dense in $W$, hence $\uppi(X)$ is Zariski dense in $\uppi(W)$.
We apply induction on $\dim \uppi(X) = \dim \uppi(W)$.
If $\dim \uppi(X) = 0$ then $\uppi(X)$ is finite, so $X$ is finite, hence Zariski closed, and we take $\ell = 1$, $V_1 = X$.
Suppose $\dim \uppi(W) \ge 1$.
Let $T = [\uppi(W_1) \setminus Y_1]\cup\cdots\cup[\uppi(W_k)\setminus Y_k]$.
By the claim and Fact~\ref{lem:dim-fin-union} we have
\begin{align*}
\dim T &= \max \{ \dim \uppi(W_1) \setminus Y_1 , \ldots , \dim \uppi(W_k) \setminus Y_k \} \\
&< \max \{ \dim \uppi(W_1) , \ldots , \dim \uppi(W_k) \} = \dim \uppi(W).
\end{align*}
As $T$ is constructible $X \cap [T \times \Aa^1]$ is definable.
Applying induction to $X \cap [T \times \Aa^1]$ we obtain closed subvarieties $V_1,\ldots,V_{\ell - 1}$ of $Z \times \Aa^1$ such that if $\alpha \in Z(K) \cap T$ and $X_\alpha\ne\emptyset$, then there is $i \in \{1,\ldots, \ell - 1\}$ such that $X_\alpha \subseteq (V_i)_\alpha$ and $(V_i)_\alpha$ is finite.
Now suppose $\alpha \in Z(K)$ and $\alpha \notin T$.
By definition of $Z$ each $(W_i)_\alpha$ is finite, hence $W_\alpha$ is finite.
Let $V_\ell = W$.
\end{proof}

\section{Theorem E generic continuity of definable functions}
\label{section:functions}

\begin{proposition}
\label{prop:gen-gen-cont}
Suppose that $K$ is \'ez, $X$ is a definable subset of $K^m$, and $f \colon X \to K^n$ is definable.
Let $E$ be the set of $a \in X$ at which $f$ is continuous.
Then $\dim X \setminus E < \dim X$.
\end{proposition}

We do not know if $E$ is definable.
Proposition~\ref{prop:gen-gen-cont} shows that the set of points of discontinuity is contained in a definable subset of $X$ of dimension $< \dim X$.
Proposition~\ref{prop:gen-gen-cont} follows from Proposition~\ref{prop:gen-cont} and Lemma~\ref{lem:ez-mod}.

\begin{proposition}
\label{prop:gen-cont}
Suppose that $K$ is \'ez, $V$ is a smooth irreducible subvariety of $\Aa^m$, $O$ is a nonempty definable \'etale open subset of $V(K)$, and $f \colon  O \to K^n$ is definable.
Then there is a dense open subvariety $U$ of $V$ such that $f$ is continuous on $O \cap U(K)$.
\end{proposition}

Thus if $K$ is \'ez then any definable function $K^m \to K^n$ is $\Sa E_K$-continuous on a dense Zariski open subset of $K^m$.
Note that $O \cap U(K)$ is $\Sa E_K$-dense in $O$ by Fact~\ref{fact:with-anand}.
Proposition~\ref{prop:gen-cont} is a consequence of the following generic description of definable functions with codomain $K$.

\begin{proposition}
\label{prop:one var function decomp}
Suppose that $K$ is \'ez, $V$ is a smooth irreducible subvariety of $\Aa^m$, $O$ is a nonempty definable \'etale open subset of $V(K)$, and $f \colon  O \to K$ is definable.
Then there is a dense open subvariety $U$ of $V$, definable \'etale open subsets $O_1,\ldots,O_k$ of $O$, and irreducible $h_1,\ldots,h_k \in K[x_1,\ldots,x_m,t]$ such that $O \cap U(K) = \bigcup_{i = 1}^{k} O_i$ and for every $i \in \{1,\ldots,k\}$:
\begin{enumerate}
\item $h_i(\alpha,f(\alpha)) = 0$ and $h_i(\alpha,t)$ is not constant zero for all $\alpha \in O_i$,
\item the closed subvariety $W_i$ of $U \times \Aa^1$ given by $h_i(x_1,\ldots,x_m,t) = 0$ is smooth,
\item the graph of the restriction of $f$ to $O_i$ is an \'etale open subset of $W_i(K)$,
\item $f$ is continuous on $O_i$.
\end{enumerate}
\end{proposition}

We prove Proposition~\ref{prop:one var function decomp} by obtaining $(1) - (3)$ and then applying Lemma~\ref{lem:cont} to get $(4)$.
We $\Gamma(f)$ be the graph of a function $f$.

\begin{lemma}
\label{lem:cont}
Suppose that $K$ is large and perfect, $V$ is a smooth irreducible $K$-variety, $O$ is a nonempty $\Sa E_K$-open subset of $V(K)$, $W$ is a smooth irreducible subvariety of $V \times \Aa^n$ with $|W_\alpha| < \infty$ for all $\alpha \in V$, and $f \colon   O \to K^n$ is such that $\Gamma(f)$ is an \'etale open subset of $W(K)$.
Then there is dense open subvariety $U$ of $V$ such that $f$ is continuous on $U(K) \cap O$.
\end{lemma}

\begin{proof}
Let $\uppi$ be the projection $W \to V$.
Then $\uppi(W)$ contains $O$, so by Fact~\ref{fact:with-anand} $\uppi(W)$ is Zariski dense in $V$.
Therefore $\uppi$ is dominant.
By Fact~\ref{fact:constructible}.5 $\dim V = \dim W$.
Corollary~\ref{cor:perfect} gives a dense open subvariety $U$ of $V$ such that the projection 
$W(K) \cap [U(K) \times K^n] \to U(K)$ is $\Sa E_K$-open.
Suppose that $a \in U(K) \cap O$.
We show that $f$ is continuous at $\alpha$.
Let $P \subseteq K^n$ be an \'etale open neighbouuprhood of $f(\alpha)$.
By Fact~\ref{fact:first-paper-system}.5 $U(K) \times P$ is an \'etale open neighbouuprhood of $(\alpha,f(\alpha))$, hence $Q := \uppi( \Upgamma(f) \cap [U(K) \times P])$ is an \'etale open neighbouuprhood of $\alpha$.
Suppose that $\alpha^* \in Q$.
The projection $\Gamma(f) \cap [U(K) \times P] \to U(K)$ is injective, so $(\alpha^*,f(\alpha^*))$ is in $\Gamma(f) \cap [U(K) \times P]$, hence $f(\alpha^*) \in P$.
\end{proof}

Lemma~\ref{lem:irreducible} produces the irreducibility required by Proposition~\ref{prop:one var function decomp}.

\begin{lemma}
\label{lem:irreducible}
Suppose that $K$ is algebraically bounded, $X$ is a definable subset of $K^m$, and $f \colon  X \to K$ is definable.
Then there are irreducible $g_1,\ldots,g_k \in K[x_1,\ldots,x_m,t]$ such that for every $\alpha \in X$ there is $i$ such that $g_i(\alpha,t)$ is not constant zero and $g_i(\alpha,f(\alpha)) = 0$.
\end{lemma}

\begin{proof}
As $K$ is algebraically bounded there are $h_1,\ldots,h_k \in K[x_1,\ldots,x_m,t]$ such that for every $\alpha \in X$ there is $i$ such that $h_i(\alpha,t)$ is not constant zero and $h_i(\alpha,f(\alpha)) = 0$.
For each $i$ let $h^1_i,\ldots,h^{\ell(i)}_k \in K[x_1,\ldots,x_m,t]$ be the irreducible factors of $h_i$.
Then for every $\alpha \in X$ there are $i,j$ such that $h_i(\alpha,t)$ is not constant zero and $h^j_i(\alpha,f(\alpha)) = 0$.
Note that $h^j_i(\alpha,t)$ cannot be constant zero.
\end{proof}

We now prove Proposition~\ref{prop:one var function decomp}.

\begin{proof}
Applying Theorem~\ref{thm:ez bounded} and Lemma~\ref{lem:decomp} we get irreducible $h_{i},\ldots,h_{k} \in K[x_1,\ldots,x_m,t]$ such that for every $\alpha \in O$ there is $i \in \{1, \ldots, \ell \}$ such that $h_{i}(\alpha,t)$ is not constant zero and $h_{i}(\alpha,f(\alpha)) = 0$.
For each $i$ let
$$ Y_{i} = \{ \alpha \in U : h_{i}(\alpha,t) \ne 0, h_{i}(\alpha,f(\alpha)) = 0\}. $$
Note that each $Y_{i}$ is definable, hence \'ez, and the $Y_{i}$ cover $O$.
Applying Lemma~\ref{lem:int-vs-nwd} we see that for each $i$ we have $Y_{i} = O_{i} \cup Y'_{i}$ where $O_{i}$ is a definable \'etale open subset of $V(K)$ and $Y'_{i}$ is not Zariski dense in $V$.
Let $U$ be a dense open subvariety of $V$ such that each $Y'_i$ is disjoint from $U$.
After replacing $O$ with $U(K) \cap O$ we may suppose that each $Y_i$ is \'etale open.
Let $W_i$ be the closed subvariety of $U \times \Aa^1$ given by $h_i(x_1,\ldots,x_m,t) = 0$, note that $W_i$ is irreducible as $h_i$ is irreducible.
The image of the projection $W_i \to U$ contains $O_i$ and is hence dominant.
For each $i$ let $U_i$ be the set of $\alpha \in V$ such that $|(W_i)_\alpha| < \infty$.
By Fact~\ref{fact:constructible} each $U_i$ is an open subvariety of $V$.
If $\alpha \in O_i$ then $|(W_i)_\alpha| < \infty$ as $h_i(\alpha,t)$ is not constant zero, so each $U_i$ is Zariski dense in $V$ by Fact~\ref{fact:with-anand}.
After possibly replacing $U$ with $U_1\cap\cdots\cap U_n$ we suppose that each projection $W_i \to U$ has finite fibers.
For each $i$ let $W'_i$ be the singular locus of $W_i$.
As $K$ is perfect $W'_i$ is a proper closed subvariety of $W_i$ so $\dim W'_i < \dim W_i$.
Let $\uppi$ be the projection $U \times \Aa^1 \to U$.
Hence 
$$ \dim \uppi(W'_i) = \dim W'_i < \dim W_i = \dim U. $$
where the equalities hold by Fact~\ref{fact:constructible} as the projection $W_i \to U$ has finite fibers.
Hence each $\uppi(W'_i)$ is not Zariski dense in $U$, so there is a nonempty open subvariety $U'$ of $U$ which is disjoint from each $\uppi(W'_i)$.
For each $i$, $W_i \cap [U' \times \Aa^1]$ is smooth, so after replacing $U$ with $U'$ we suppose that each $W_i$ is smooth.
We maintain our assumption that each $W_i$ is irreducible as an open subvariety of an irreducible variety is irreducible.

\meno
It remains to arrange that the graph of the restriction of $f$ to $O_i$ is an \'etale open subset of $W_i(K)$.
Let $f_i$ be the restriction of $f$ to $O_i$.
Then $\Gamma(f_i)$ is an \'ez subset of $W_i(K)$, so by Lemma~\ref{lem:int-vs-nwd} $\Gamma(f_i) = P_i \cup Z_i$ where $P_i$ is a definable \'etale open subset of $W_i(K)$ and $Z_i$ is not Zariski dense in $W_i$.
Let $Z'_i$ be the Zariski closure of $Z_i$ in $W_i$.
As above we have $\dim \uppi(Z'_i) = \dim Z'_i < \dim W_i = \dim U$.
After again shrinking $U$ as above we suppose that $U$ is disjoint from each $\uppi(Z'_i)$.
It follows that $\Gamma(f_i) = P_i$ for all $i$.
\end{proof}

We now prove Proposition~\ref{prop:gen-cont}

\begin{proof}
Let $f = (f_1,\ldots,f_n)$.
Applying Proposition~\ref{prop:one var function decomp} we obtain for each $i \in \{1,\ldots,n\}$ a dense open subvariety $U_i$ of $V$, irreducible polynomials $h_{i1},\ldots,h_{i\ell} \in K[x_1,\ldots,x_m,t]$, and definable \'etale open subsets $O_{i1},\ldots,O_{i\ell}$ of $O$ such that for each $i$:
\begin{enumerate}
\item $O \cap U_i(K) = \bigcup_{j = 1}^{\ell} O_{ij}$,
\item $h_{ij}(\alpha,f_i(\alpha)) = 0$ and $h_{ij}(\alpha,t)$ is non-constant zero for all $\alpha \in O_{ij}$,
\item the graph of the restriction of $f_i$ to $O_{ij}$ is an \'etale open subset of $W_{ij}(K)$, where $W_{ij}$ is the closed subvariety of $U_i \times \Aa^1$ given by $h_{ij}(x_1,\ldots,x_m,t) = 0$.
\end{enumerate}

Let $U = \bigcap_{i = 1}^{n} U_i$, then $U$ is a dense open subvariety of $V$.
After replacing each $O_{ij}$ with $O_{ij} \cap U(K)$ we suppose $U(K)$ contains every $O_{ij}$.
For each $\upsigma \colon  \{1,\ldots,n\} \to \{1,\ldots,\ell\}$ let $O_\upsigma$ be $\bigcap_{i = 1}^{n} O_{i\upsigma(i)}$.
Note that $O \cap U(K)$ is the union of the $O_\upsigma$.
It is enough to show that for every $\upsigma$ there is a dense open subvariety $U_\upsigma$ of $V$ such that $f$ is continuous on $O_\upsigma \cap U_\upsigma(K)$.
Hence we fix such $\upsigma$ such that $O_\upsigma$ is nonempty, let  $O = O_\upsigma$ and $h_i = h_{i\upsigma(i)}$.
For each $i$ let $W_i$ be the closed subvariety of $U \times \Aa^1$ given by $h_i(x_1,\ldots,x_m,t) = 0$.
Then the graph of the restriction of each $f_i$ to $O$ is an \'etale open subset of $W_i(K)$.
Following the argument of Proposition~\ref{prop:one var function decomp} we may also suppose that $|(W_i)_a| < \infty$ for all $a \in U$ and $i \in \{1,\ldots,n\}$.

\meno
Now let $W$ be the closed subvariety of $U \times \Aa^m$ given by
$$ h_1(x_1,\ldots,x_m,t) = \cdots = h_n(x_1,\ldots,x_m,t) = 0. $$
For each $i \in \{1,\ldots,m\}$ let $\uppi_i \colon  U \times \Aa^m \to U \times \Aa^1$ be given by $\uppi_i(x,y_1,\ldots,y_m) = (x,y_i)$ and let $\uprho_i \colon  U(K) \times K^m \to U(K) \times K$ be the induced map on $K$-points.
Then
$$ W = \uppi_1^{-1}(W_1) \cap \cdots \cap \uppi_n^{-1}(W_n) \quad \text{and} \quad \Gamma(f) = \uprho_1^{-1} (\Gamma(f_1)) \cap \cdots \cap \uprho_n^{-1}(\Gamma(f_n))$$
Note that each $\uppi_i^{-1}(W_i)$ is a closed subvariety of $U \times \Aa^m$ and each $\uprho_i^{-1}(\Gamma(f_i))$ is an \'etale open subset of $\uppi_i^{-1}(W_i)(K)$.
Therefore $\Gamma(f)$ is an \'etale open subset of $W(K)$.
Note also that $|W_\alpha| < \infty$ for all $\alpha \in U$.
The proposition now follows by an application of Lemma~\ref{lem:cont}.
\end{proof}

We finally proof Proposition~\ref{prop:gen-gen-cont}.

\begin{proof}
Applying Lemma~\ref{lem:ez-mod} let $U$ be a dense open subvariety of $\Aa^m$, $V$ be a smooth subvariety of $\Aa^m$, and $O$ be a definable \'etale open subset of $V(K)$ such that $X \cap U(K) = O$ and $\dim X \setminus O < \dim X$.
Let $V_1,\ldots,V_k$ be the irreducible components of $V$.
Applying Proposition~\ref{prop:gen-cont} we fix for each $i$ a dense open subvariety $U_i$ of $V_i$ such that $f$ is continuous on each $X \cap U_i(K)$.
Note that $E$ contains $\bigcup_{i = 1}^{k} X \cap U_i(K)$ and $\dim X \setminus \bigcup_{i = 1}^{k} U_i(K) < \dim X$.
\end{proof}

\bibliographystyle{amsalpha}
\bibliography{ref}

\end{document}